\documentclass[11pt]{article}
\usepackage{amsmath}
\usepackage{amsfonts}
\usepackage{amssymb}
\usepackage{graphicx, times}
\usepackage{subfig}
\usepackage{amsthm}
\usepackage{bm}
\usepackage[left=2.5cm, right=3cm, top=2cm]{geometry}
\usepackage[section]{placeins}
\usepackage[section]{placeins}  
\newtheorem{The}{Theorem}[section]
\newtheorem{Def}{Definition}[section]
\newtheorem{lem}{Lemma}[section]
\newtheorem{Ex}{Example}[section]

\begin{document}
\begin{center}
{\LARGE {\bf Analysis of solution trajectories of linear fractional order systems }}
\vskip 1cm
{\Large   Madhuri Patil, Sachin Bhalekar}\\
\textit{Department of Mathematics, Shivaji University, Kolhapur - 416004, India, Email:madhuripatil4246@gmail.com (Madhuri Patil),sachin.math@yahoo.co.in, sbb\_maths@unishivaji.ac.in (Sachin Bhalekar), }\\
\end{center}
\begin{abstract}
 The behavior of solution trajectories usually changes if we replace the classical derivative in a system by a fractional one. In this article, we throw a light on the relation between two trajectories $X(t)$ and $Y(t)$ of such a system, where the initial point $Y(0)$ is at some point $X(t_1)$ of trajectory $X(t)$. In contrast with classical systems,  trajectories $X$ and $Y$ do not follow the same path. Further, we provide a Frenet apparatus of both trajectories in various cases and discuss their effect.
\end{abstract}
\vskip 0.5cm
\noindent
{\bf Keywords}: Fractional derivative, Mittag-Leffler functions, Orthogonal transformation, Frenet apparatus.

\section{Introduction}
\par In the recent past, fractional differential equations (FDE) became a popular topic among the researchers working in pure as well as applied Mathematics. Applications of FDEs are found in various fields ranging from Physics to Biology. We suggest some selected references \cite{Mainardi,Kulish,Fellah,Matusu,El-Saka,Yuan,Goulart,Sebaa,Magin} on applications of FDEs to the readers.
\par Mathematical analysis of FDEs is also an interesting and equally important topic of research. Existence and uniqueness \cite{Delbosco, Diethelm, Gejji1, Wei}, stability \cite{Matignon1,Matignon2,Moze,Deng1,Deng2,Qian,Agarwal} and positivity \cite{Zhang,Gejji2,Gejji3,Bai,Goodrich,Baleanu} of these equations is studied by the researchers in details. Fractional order versions of stable manifold theorem are discussed in \cite{Cong1,Deshpande,Cong2}. Since FDEs are generalizations of classical differential dynamical systems, we cannot expect the same properties from these models as the classical ones.
\par In \cite{S. Bhalekar}, we have shown that the planar linear FDE system  ${}_0^C\mathrm{D}_t^\alpha X=AX$ may produce self-intersecting trajectories. Such singular points are not shown by their classical counterparts. We continue our investigations in the present manuscript and discuss the trajectories of FDE systems whose initial point is on a different trajectory of the same system.

\section{Preliminaries}
This section contains basic definitions and results given in the literature.  
\begin{Def}\cite{Podlubny} \label{Def 3,2.1}
Let $\alpha\ge0$ \,\, ($\alpha\in\mathbb{R}$). Then Riemann-Liouville (\text RL) fractional integral of function $f\in C[0,b]$, $t>0$ of order `$\alpha$' is defined as,
\begin{equation}
{}_0\mathrm{I}_t^\alpha f(t)=
\frac{1}{\Gamma{(\alpha)}}\int_0^t (t-\tau)^{\alpha-1}f(\tau)\,\mathrm{d}\tau. \label{3.1}
\end{equation}
\end{Def}
\begin{Def}\cite{Podlubny}\label{Def 3,2.2}
The Caputo fractional derivative of order $\alpha>0$, $n-1<\alpha< n$, $n\in \mathbb{N}$ is defined for $f\in C^n[0,b]$,\, $t>0$ as,
\begin{equation}
{}_0^{C}\mathrm{D}_t^\alpha f(t)=
\begin{cases}
\frac{1}{\Gamma{(n-\alpha)}}\int_0^t (t-\tau)^{n-\alpha-1}f^{(n)}(\tau)\,\mathrm{d}\tau & \mathrm{if}\,\, n-1<\alpha< n\\
\frac{d^n}{dt^n}f(t) & \mathrm{if}\,\, \alpha=n.
\end{cases}\label{3.2}
\end{equation}
Note that ${}_0^{C}\mathrm{D}_t^\alpha c=0$, where $c$ is a constant.
\end{Def}
\begin{Def}\cite{Podlubny} \label{Def 3,2.3}
The one-parameter Mittag-Leffler function is defined as,
\begin{equation}
E_\alpha(z)=\sum_{k=0}^\infty \frac{z^k}{\Gamma(\alpha k+1)}\, ,\qquad z\in\mathbb{C}, \,\,(\alpha>0).\label{3.3}
\end{equation}
The two-parameter Mittag-Leffler function is defined as,
\begin{equation}
E_{\alpha,\beta}(z)=\sum_{k=0}^\infty \frac{z^k}{\Gamma(\alpha k+\beta)}\, ,\qquad z\in\mathbb{C}, \,\,(\alpha>0,\,\beta>0).\label{3.4}
\end{equation}
\end{Def}
\begin{Def}\cite{O'Neill}\label{def 3,2.4}
Let $\alpha: I \rightarrow \mathbb{R}^n$ \,be a curve. The speed of $\alpha$ is defined as \begin{equation}
\nu(t) =\parallel \alpha'(t) \parallel. \label{3.5}
\end{equation}
\end{Def}
\begin{Def}\cite{O'Neill}\label{def 3,2.5}
An isometry of \,$\mathbb{R}^n$ is a mapping $F:\mathbb{R}^n\rightarrow\mathbb{R}^n$ such that 
\begin{equation}
d(F(p),F(q))=d(p,q) \label{3.6}
\end{equation}
for all points $p, q$ in $\mathbb{R}^n$. d(x,y) is Euclidean distance.
\end{Def}
\begin{Def}\cite{O'Neill}\label{def 3,2.6}
Two curves  $\alpha, \beta:I\rightarrow\mathbb{R}^n$ are congruent provided there exists an isometry $F$ of\, $\mathbb{R}^n$ such that $\beta=F(\alpha)$; that is, $\beta(t)=F(\alpha(t))$ for all $t$ in $I$. 
\end{Def}
\begin{Def}\cite{O'Neill}\label{def 3,2.7}
A transformation $C:\mathbb{R}^n\rightarrow \mathbb{R}^n$ is an Orthogonal transformation if it preserves dot products in the sense that 
\begin{equation}
C(p)\cdot C(q)=p\cdot q \qquad for\, all\,\, p,q. \label{3.7}
\end{equation}
Every orthogonal transformation is an isometry.
\end{Def}

\begin{The} \label{Thm 3,2.1}
\cite{Luchko} Solution of homogeneous system of  fractional order differential equation
\begin{equation}
{}_0^C\mathrm{D}_t^\alpha X(t)=A X(t), \qquad 0<\alpha<1 \label{3.8}
\end{equation} 
(where $A$ is $n\times n$ matrix) is given by
\begin{equation}
X(t)=E_\alpha(A t^\alpha)X(0),\label{3.9}
\end{equation}
 where $E_\alpha(A t^\alpha)$ is matrix variate Mittag-Leffler function.
\end{The}
\begin{The}\cite{O'Neill}\label{Thm 3,2.2}
For planar regular curve $\alpha:I\rightarrow\mathbb{R}^2$ given by $\alpha(t)=(x(t),y(t)),\, t\in I$, the Frenet apparatus is given by
\begin{equation}
\begin{split}
T & = \frac{(\dot{x},\dot{y})}{\sqrt{(\dot{x})^2+(\dot{y})^2}}\\
N & = \frac{(-\dot{y},\dot{x})}{\sqrt{(\dot{x})^2+(\dot{y})^2}} \\
 k & = \frac{\mathrm{det}\begin{vmatrix}
 \dot{x} & \dot{y}\\
 \ddot{x} & \ddot{y}
 \end{vmatrix}}{((\dot{x})^2+(\dot{y})^2)^{3/2}}.
\end{split} \label{3.10}
\end{equation}
\end{The}

\section{Observations}
We have following observations.\\
{\bf (1)} Consider the system 
\begin{equation}
\dot{X}(t)=
\begin{bmatrix}
-2 & 4\\
-4 & -2
\end{bmatrix}
X(t). \label{3.11}
\end{equation}
Solution of the linear system (\ref{3.11}) with initial condition $X(0)=[1, 1]^T$ is given in the Figure \ref{Fig 1} and it is shown by a blue line. 
\par Now, consider the same system (\ref{3.11}) with initial condition $X(0)=[e^{-1}(\cos2 + \sin2), e^{-1}(-\sin2+\cos2)]^T$ on the original trajectory,  discussed above. Solution of this system is shown in the same figure by a red line.\\
It can be observed that both the trajectories follow the same path.
\begin{figure}[h]
\begin{center}
          \includegraphics[width=0.5\textwidth]{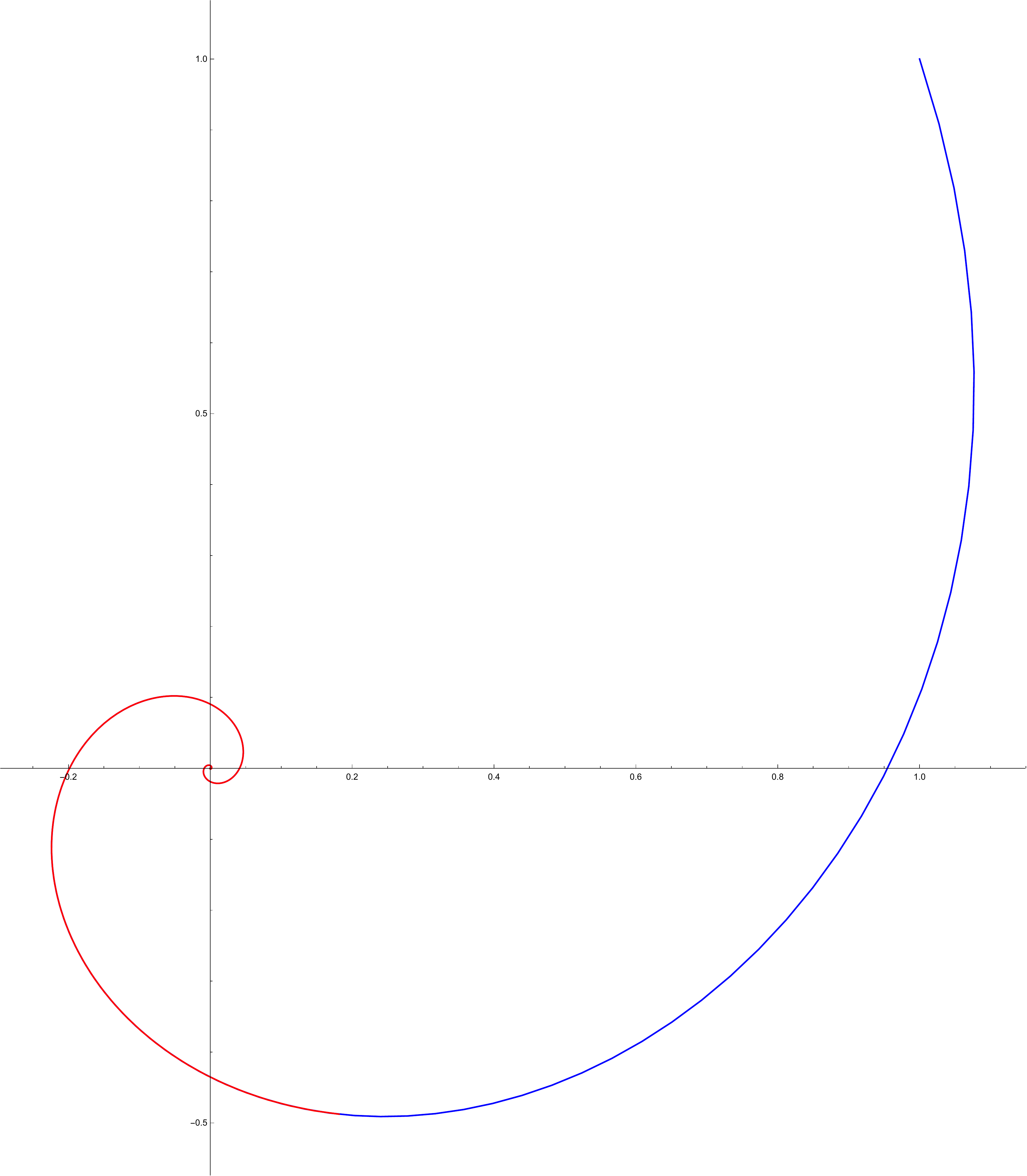}
                 \caption{Solution trajectory of system (\ref{3.11}) starting at some point on original trajectory.}
                 \label{Fig 1}
\end{center}       
        \end{figure}        
 
 {\bf (2)} Consider the non-autonomous system of differential equations 
 \begin{equation}
 \dot{X}(t)=\begin{bmatrix}
 6t\\ 3t^2-3
 \end{bmatrix}.\label{3.12}
 \end{equation}
 The solution trajectories of this system with distinct initial points (as above) are shown in Figure \ref{Fig 2}. It can be observed that (cf. blue curve in Figure \ref{Fig 2}), the loop in the original trajectory can be eliminated by choosing the initial condition of a new trajectory at a point on the original trajectory after self-intersection.   
 \begin{figure}[h]
 \begin{center}
           \includegraphics[width=0.5\textwidth]{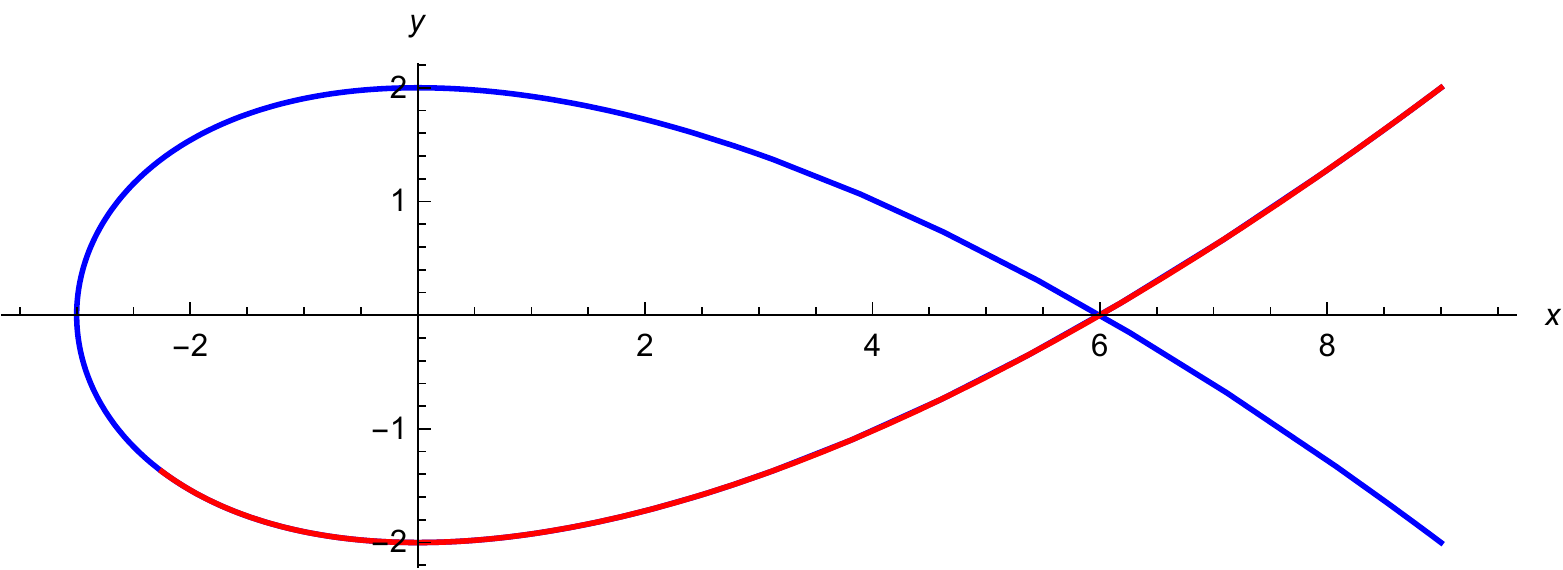}
          \caption{Self-intersecting trajectory of system (\ref{3.12}).}
         \label{Fig 2}
 \end{center}  
         \end{figure}

{\bf (3)} Consider fractional order system 
 \begin{equation}
 {}_0^C\mathrm{D}_t^{0.7} X(t)=
 \begin{bmatrix}
 -1 & 3\\
 -3 & -1
 \end{bmatrix}
 X(t). \label{3.13}
 \end{equation}
 In Figure \ref{Fig 3}, we show solutions to this system with different initial conditions. As in the last case, the initial condition of the second system is at some point on the original trajectory. However, the paths followed by these two trajectories are different, unlike in classical model (\ref{3.11}).
  \begin{figure}[h]
   \begin{center}
             \includegraphics[width=0.3\textwidth]{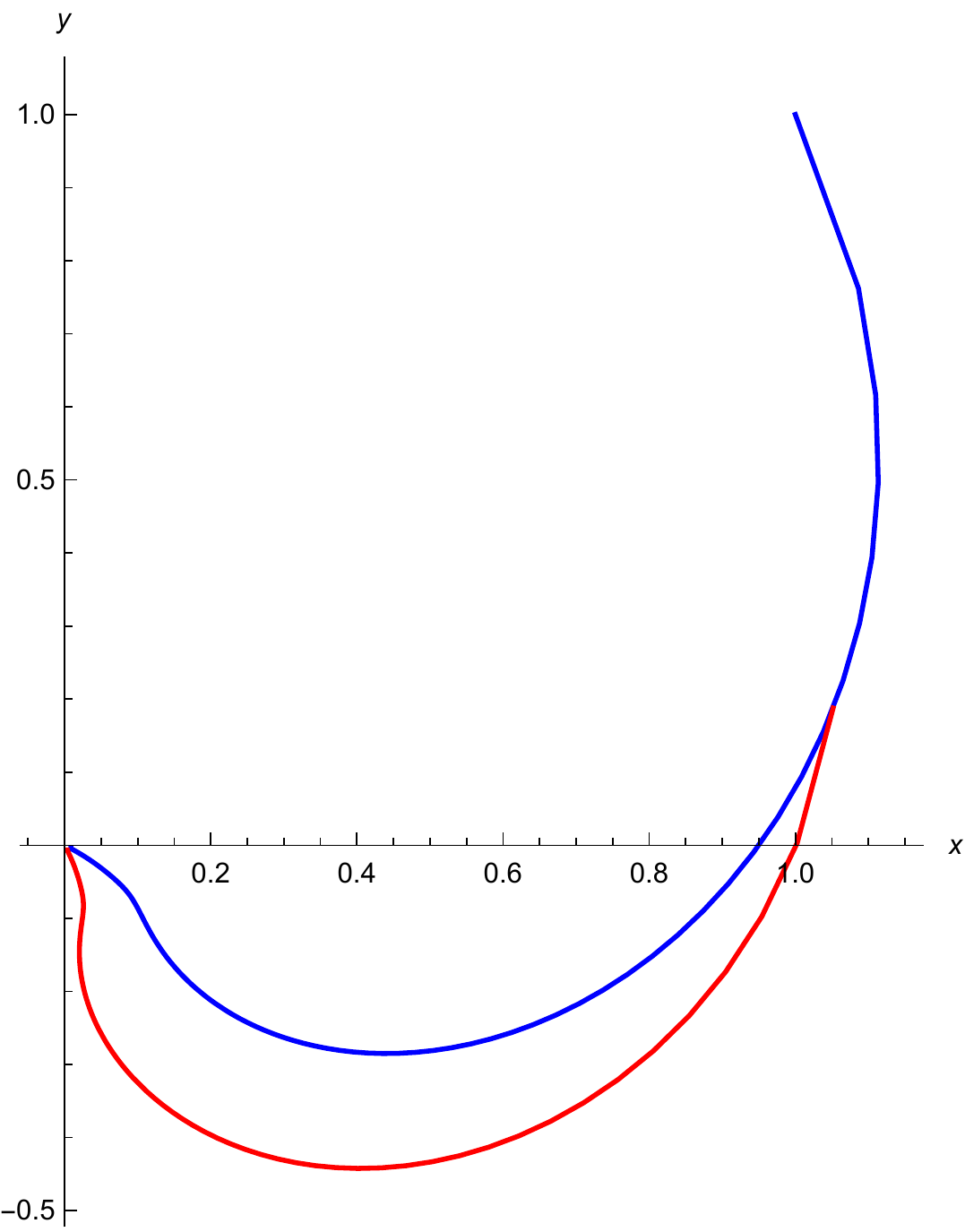}
            \caption{solution trajectory of system (\ref{3.13}).}         
   \label{Fig 3}
   \end{center}  
          \end{figure}  

 {\bf (4)} In paper \cite{S. Bhalekar}, we have observed self-intersecting trajectories of some linear fractional order systems. Consider the system,
 \begin{equation}
  {}_0^C\mathrm{D}_t^{0.1} X(t)=
  \begin{bmatrix}
 0.983469 & 0.181075\\
  -0.181075 & 0.983469
  \end{bmatrix}
  X(t). \label{3.14}
  \end{equation}
  If $X(0)=[1, 1]^T$, the solution trajectory shows self-intersection (see Figure \ref{Fig 4}(a)). Let us consider this system with initial condition $X(0)=[Re(E_{0.1} (0.983469 + 0.181075 i)50^{0.1}) + Im(E_{0.1} (0.983469 + 0.181075 i)50^{0.1}), -Im(E_{0.1} (0.983469 + 0.181075 i)50^{0.1}) + Re(E_{0.1} (0.983469 + 0.181075 i)50^{0.1})]^T$ on the original trajectory.
  \par Though we have taken new initial condition on the original trajectory at a point after self-intersection, the singular points cannot be removed unlike in classical case (2). Further, it seems that the new trajectory is some linear transformation of the original one.
  \begin{figure}[h]
              \subfloat[$ x(0)=y(0)=1$.]{\includegraphics[width=0.45\textwidth]{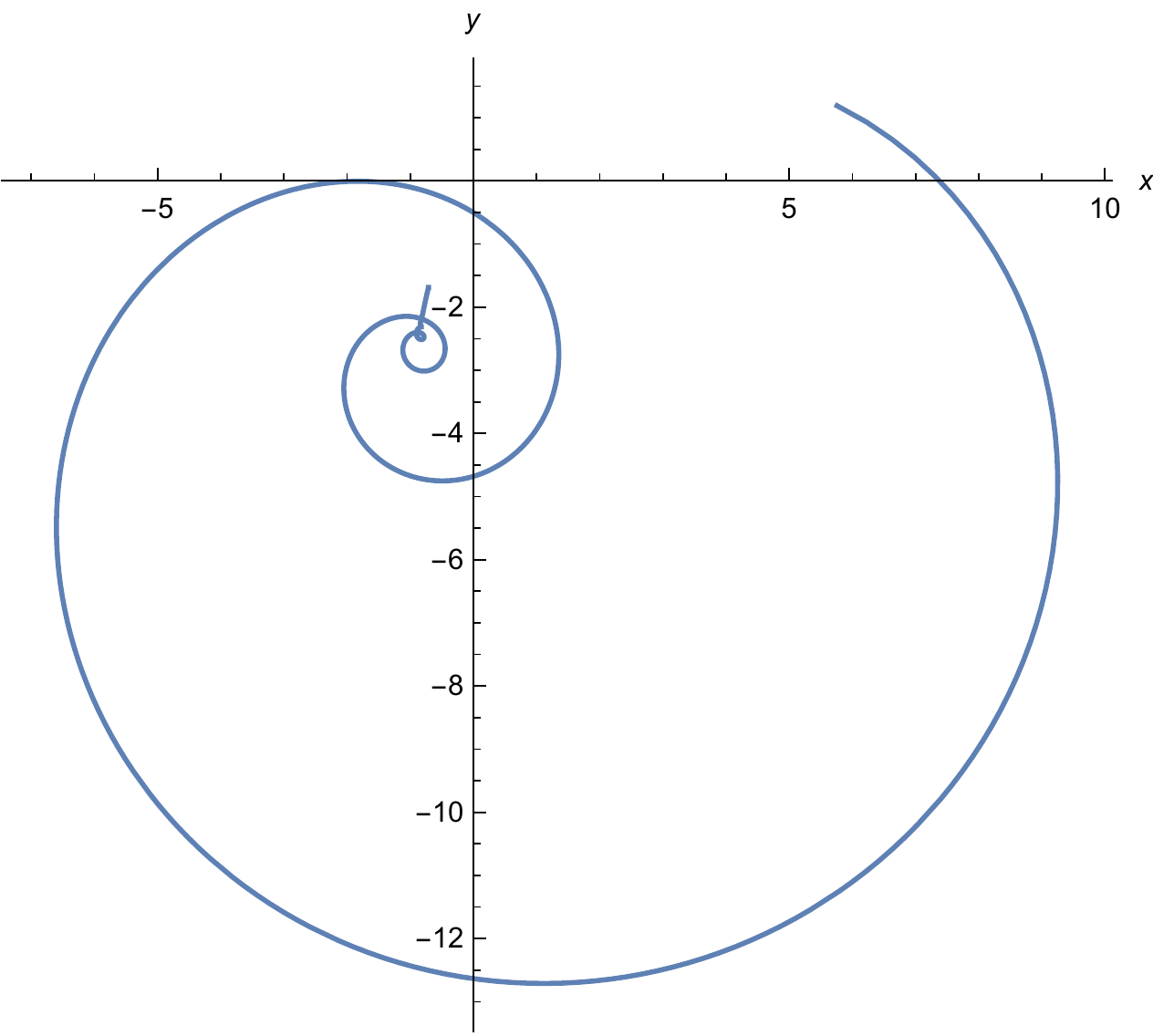}}
              \hfill 
              \subfloat[x(0) = -0.777226, 
          y(0) = -1.98038.]{\includegraphics[width=0.4\textwidth]{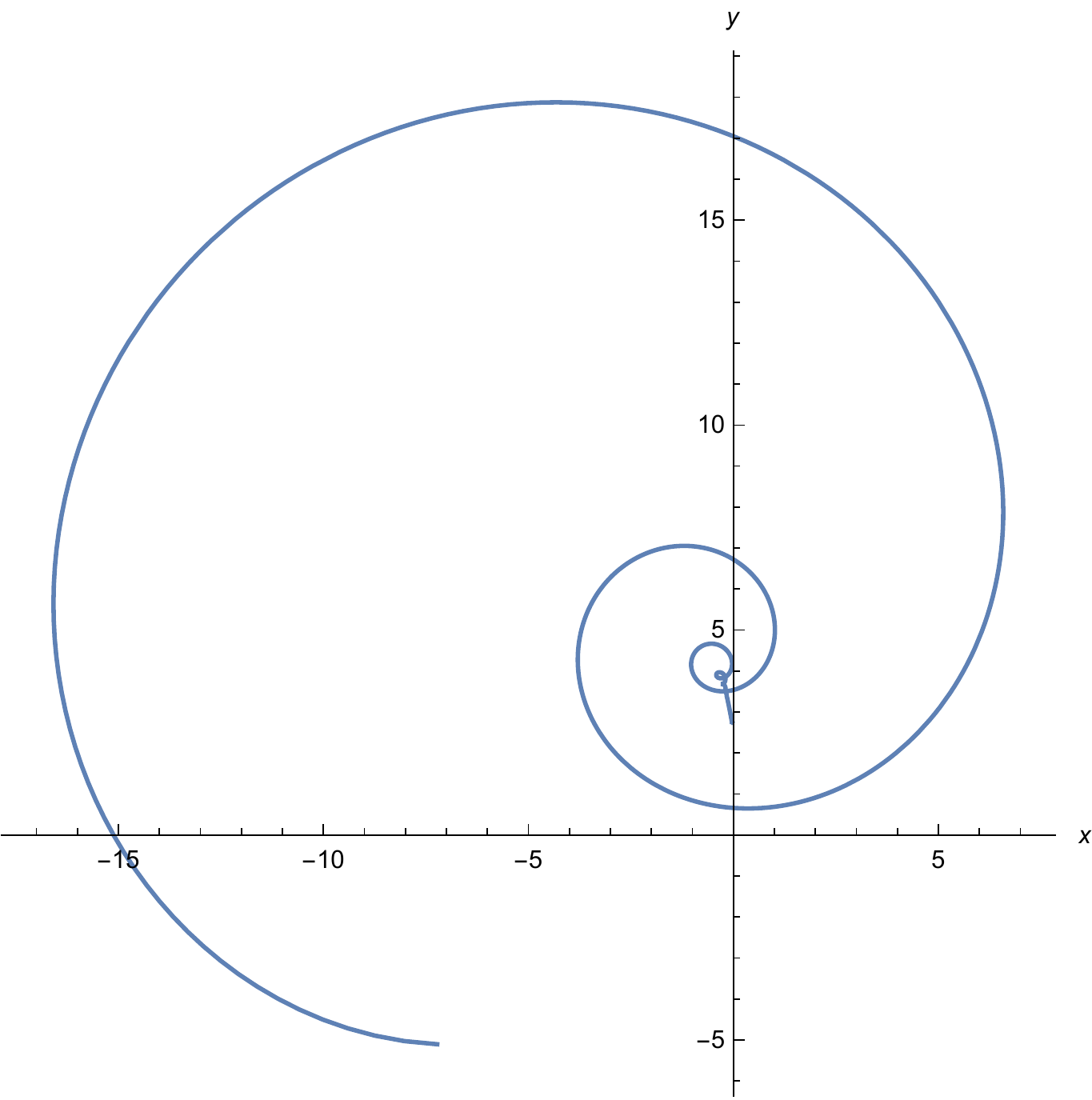}}
              \caption{Transformation of self-intersecting trajectory of system  (\ref{3.14}).
              } \label{Fig 4}
            \end{figure}   

{\bf (5)} Time used to complete the loop:\\
Consider the system (\ref{3.14}) with initial condition $X(0)=[1, 1]^T$. Consider a node formed by solution trajectory in the time interval $(12.35, 34)$. If we solve the system (\ref{3.14}) with initial condition $X(0)=[Re(E_{0.1} (0.983469 + 0.181075 i)50^{0.1}) + Im(E_{0.1} (0.983469 + 0.181075 i)50^{0.1}), -Im(E_{0.1} (0.983469 + 0.181075 i)50^{0.1}) + Re(E_{0.1} (0.983469 + 0.181075 i)50^{0.1})]^T$ on the original trajectory, then we get a node in the new trajectory in the same time interval $(12.35, 34)$. However, it seems that the new node has different size and is obtained by rotating the original node as shown in  Figure \ref{Fig 5}.  Further, the time taken to complete the loop in both the nodes is  same but the speed is different.
 \begin{figure}[h]
 \begin{center}
   \includegraphics[width=1.0\textwidth]{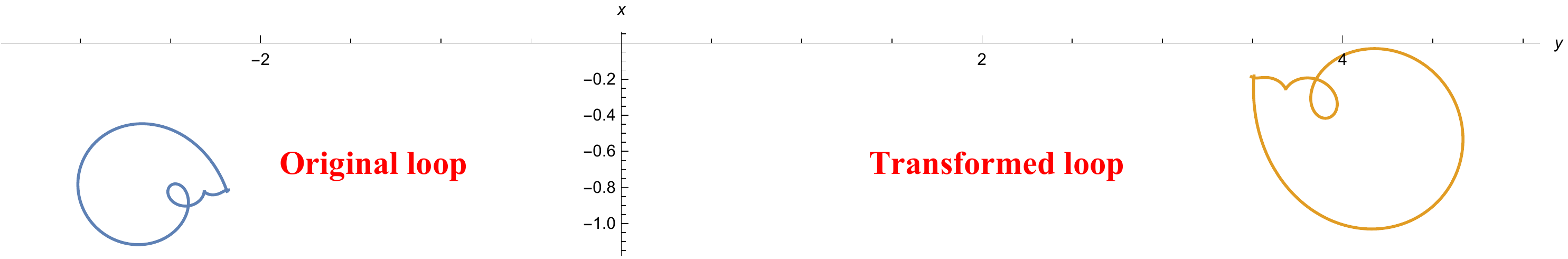}
   \caption{Transformed loops in the system  (\ref{3.14}).}
 \label{Fig 5}
    \end{center}    
            \end{figure}               
\par Our motivation for the present study is to find the linear transformation between original trajectory and the new trajectory of the fractional order system.                  

\section{Analysis}
In this section, we consider linear system of integer and fractional differential equations in $\mathbb{R}^2$ and $\mathbb{R}^3$. First we solve the system with initial condition $X(0)=X_0$ to obtain the solution $X(t)$. Then we solve the same system with initial condition at $X(t_1)$ for some $t_1>0$ and call the new solution as $Y(t)$. We show that $Y(t)=TX(t)$, where $T$ is some linear transformation.
\begin{lem}\label{Lem 3,4.1}
Consider a planar system $\dot{X}(t)=AX(t)$.\\
New trajectory $Y(t)$ starting at some point $X(t_1)$ on original trajectory $X(t)$ is given by the linear transformation
\begin{equation}
Y(t)=TX(t),\label{3.15}
\end{equation}
where $T=e^{At_1}$.\\
(i) If $A$ has real-distinct eigenvalues then $T$ represents scaling (only).\\
(ii) If $A$ has complex conjugate eigenvalues $a\pm ib$ then $T$ represents both scaling and rotation.
\end{lem}
\begin{proof}
Solution of a system of ODEs $\dot{X}(t)=AX(t)$, \, $X(0)=X_0$ is given by
\begin{equation*}
X(t)=e^{At}X_0.
\end{equation*}
Now, let us consider the system $\dot{Y}(t)=AY(t)$, \, $Y(0)=X_1$, where $X_1=X(t_1)=e^{At_1}X_0$. Then its solution is given by,
$
Y(t)  = e^{At}X_1
= e^{At_1}e^{At}X_0
 = TX(t), 
$
where $T=e^{At_1}$.\\
The qualitative behavior of the system $X'(t)=AX(t)$ does not change if we replace $A$ by its canonical form.\\
(i) If  
$A=
\begin{bmatrix}
\lambda_1 & 0\\
0 & \lambda_2
\end{bmatrix}
$, then 
\begin{equation*}
T= e^{At_1}  = \begin{bmatrix}
e^{\lambda_1t_1} & 0\\
 \\
0 & e^{\lambda_2t_1} 
\end{bmatrix}.
\end{equation*}
Here $T$ represents scaling only.\\
The type of scaling depends on sign of $\lambda_j$, $j=1,2$.
\\(ii) If 
$A=
\begin{bmatrix}
a & b\\
-b & a
\end{bmatrix}
$, then 
\begin{equation*}
\begin{split}
T= e^{At_1} & = \begin{bmatrix}
e^{at_1}\cos(bt_1) & e^{at_1}\sin(bt_1)\\
& \\
-e^{at_1}\sin(bt_1) & e^{at_1}\cos(bt_1) 
\end{bmatrix}\\[0.05cm]
& = \begin{bmatrix}
e^{at_1} & 0\\
& \\
0 & e^{at_1} 
\end{bmatrix}
\begin{bmatrix}
\cos(bt_1) & \sin(bt_1)\\
& \\
-\sin(bt_1) & \cos(bt_1) 
\end{bmatrix}\\
& = U\cdot V
\end{split}
\end{equation*}
where $U=\begin{bmatrix}
e^{at_1} & 0\\
& \\
0 & e^{at_1} 
\end{bmatrix}$ is scaling matrix and $V=\begin{bmatrix}
\cos(bt_1) & \sin(bt_1)\\
& \\
-\sin(bt_1) & \cos(bt_1) 
\end{bmatrix}$ is rotation matrix.
\end{proof}
\noindent {\bf Comment:} Scaling factor depends on real part of eigenvalue whereas 
imaginary part of eigenvalue represents angle of rotation. The curves $X(t)$ and $U^{-1}Y(t)$ are congruent.
\begin{Ex}\label{Ex 3,4.1}
Consider the two classical systems
 \begin{equation}
\dot{X}(t)=\begin{bmatrix}
-1 & 0\\
1 & -2
\end{bmatrix}X(t)\label{3.16}
\end{equation} 
and
 \begin{equation}
\dot{X}(t)=\begin{bmatrix}
0 & 1\\
-4 & 0
\end{bmatrix}X(t).\label{3.17}
\end{equation} 
\end{Ex}
In Figure \ref{Fig 6} (a) and \ref{Fig 6} (b) we sketch the solutions of system (\ref{3.16}) and (\ref{3.17}) respectively with initial conditions $X(0)=[
1, 1
]^T$ (Blue color) and $Y(0)=X(1)$ (Red color). It can be checked that both the trajectories follow the same path.
\begin{figure}[h]
               \subfloat[Solution trajectories of system (\ref{3.16})]{\includegraphics[width=0.45\textwidth]{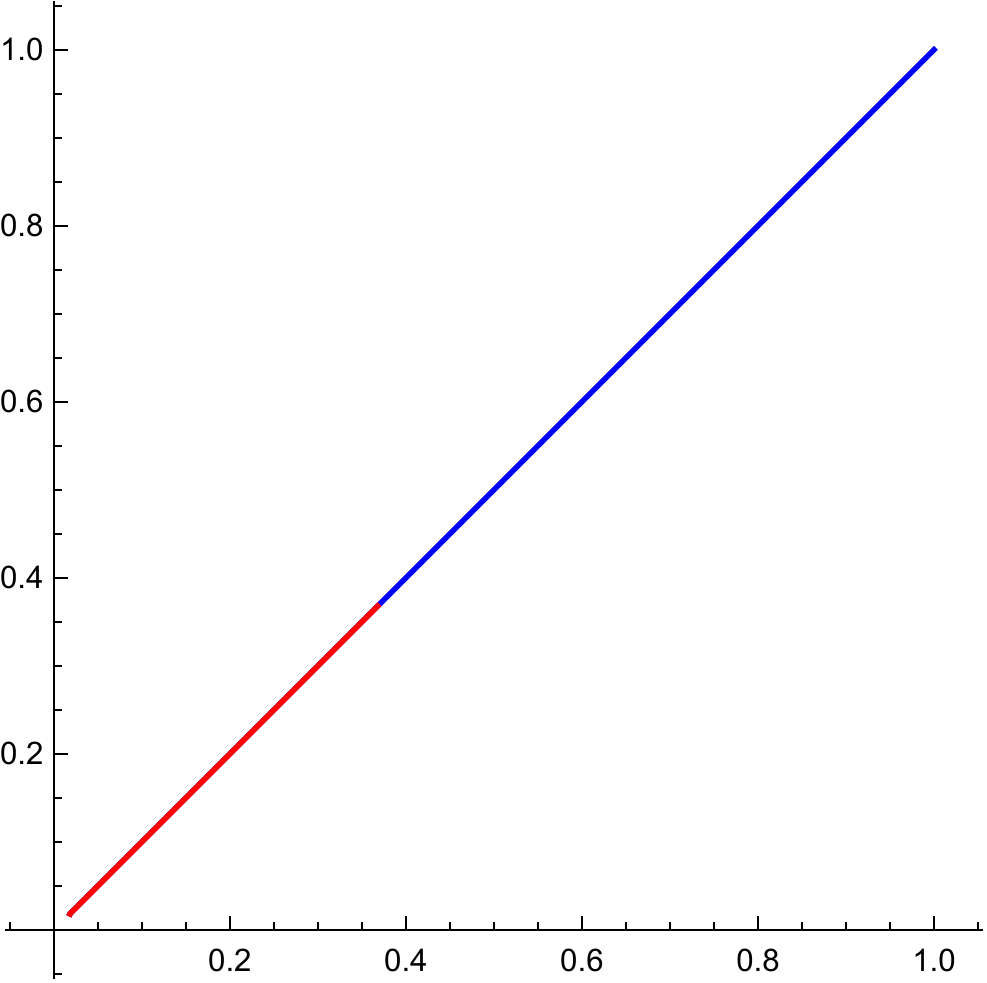}}
               \hfill 
               \subfloat[Solution trajectories of system (\ref{3.17})]{\includegraphics[width=0.3\textwidth]{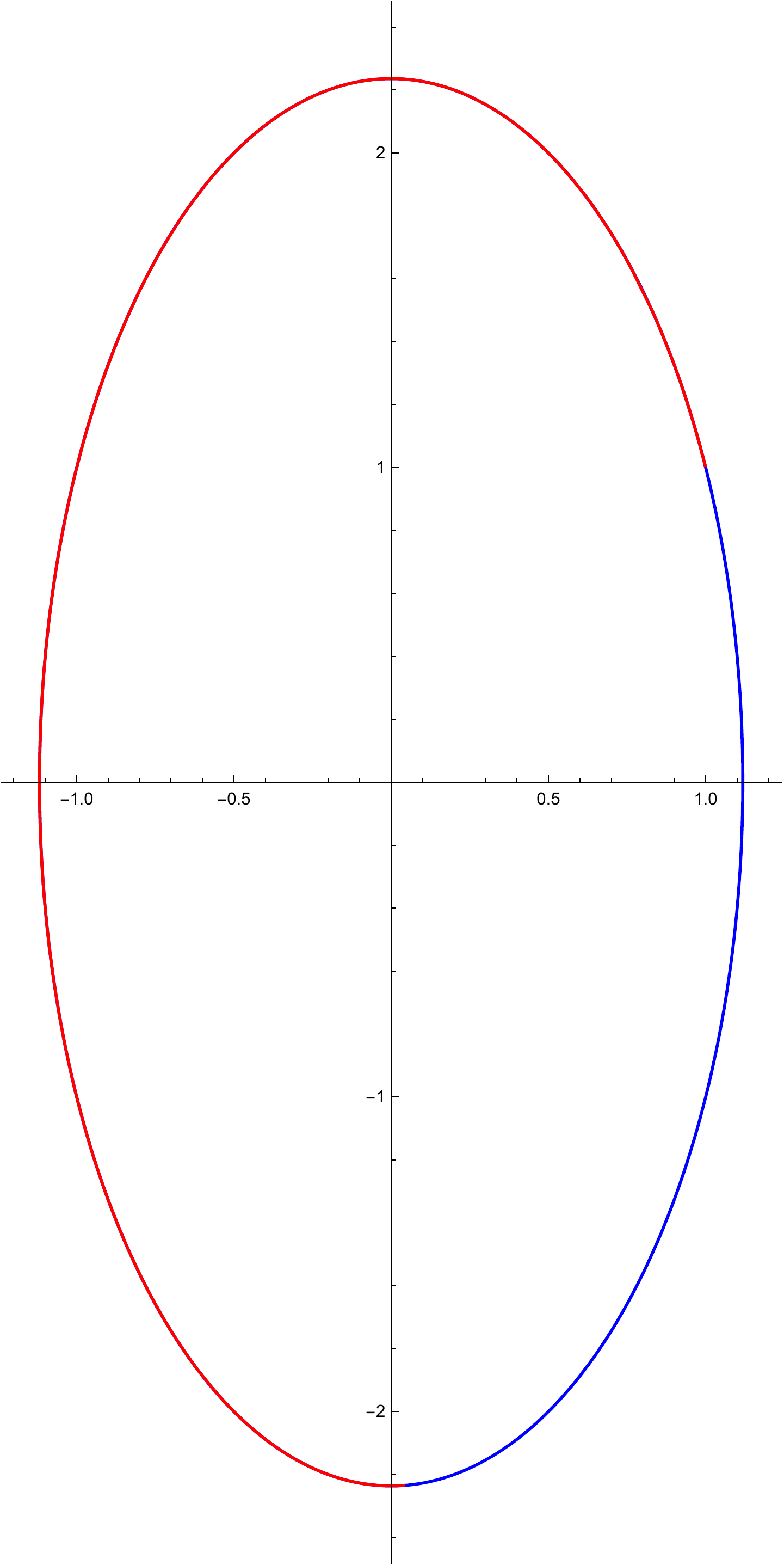}}
               \caption{
               } \label{Fig 6}
             \end{figure}  
\begin{lem}\label{Lem 3,4.2}
Consider a planar system ${}_0^C\mathrm{D}_t^\alpha X(t)=AX(t)$,\, $0<\alpha<1$.\\
New trajectory $Y(t)$ starting at some point $X(t_1)$ on original trajectory $X(t)$ is given by the linear transformation
\begin{equation}
Y(t)=TX(t),\label{3.18}
\end{equation}
where $T=E_\alpha(At_1^\alpha)$.\\
(i) If $A$ has real-distinct eigenvalues then $T$ represents scaling (only).\\
(ii) If $A$ has complex conjugate eigenvalues $a\pm ib$ then $T$ represents both scaling and rotation.
\end{lem}
\begin{proof}
 Solution of a system of FDEs \, ${}_0^C\mathrm{D}_t^\alpha X(t)=AX(t)$, \, $0<\alpha<1$, \, $X(0)=X_0$ is given by
\begin{equation*}
X(t)=E_\alpha(At^\alpha)X_0.
\end{equation*}
Now, let us consider the system ${}_0^C\mathrm{D}_t^\alpha Y(t)=A Y(t)$, \, $Y(0)=X_1$, where $X_1=X(t_1)=E_\alpha(At_1^\alpha)X_0$. Then its solution is given by,
$
Y(t)  = E_\alpha(At^\alpha)X_1
= E_\alpha(At_1^\alpha)E_\alpha(At^\alpha)X_0
 = TX(t), 
$
where $T=E_\alpha(At_1^\alpha)$. As in Lemma \ref{Lem 3,4.1}, we assume that A is in canonical form.\\
(i) If  
$A=
\begin{bmatrix}
\lambda_1 & 0\\
0 & \lambda_2
\end{bmatrix}
$, then 
\begin{equation*}
T= E_\alpha(At_1^\alpha)  = \begin{bmatrix}
E_\alpha(\lambda_1 t_1^\alpha) & 0\\
 \\
0 & E_\alpha(\lambda_2 t_1^\alpha) 
\end{bmatrix}.
\end{equation*}
Here $T$ is a scaling matrix.
\\(ii) If 
$A=
\begin{bmatrix}
a & b\\
-b & a
\end{bmatrix}
$ then 
\begin{equation*}
\begin{split}
T= e^{At_1} & = \begin{bmatrix}
Re[E_\alpha ((a+ib)t_1^\alpha)] & Im[E_\alpha ((a+ib)t_1^\alpha)]\\
& \\
-Im[E_\alpha ((a+ib)t_1^\alpha)] & Re[E_\alpha ((a+ib)t_1^\alpha)]
\end{bmatrix}\\[0.05cm]
& = \begin{bmatrix}
|E_\alpha ((a+ib)t_1^\alpha)| & 0\\
& \\
0 & |E_\alpha ((a+ib)t_1^\alpha)| 
\end{bmatrix}
\begin{bmatrix}
\frac{Re[E_\alpha ((a+ib)t_1^\alpha)]}{|E_\alpha ((a+ib)t_1^\alpha)|} & \frac{Im[E_\alpha ((a+ib)t_1^\alpha)]}{|E_\alpha ((a+ib)t_1^\alpha)|}\\
& \\
-\frac{Im[E_\alpha ((a+ib)t_1^\alpha)]}{|E_\alpha ((a+ib)t_1^\alpha)|} & \frac{Re[E_\alpha ((a+ib)t_1^\alpha)]}{|E_\alpha ((a+ib)t_1^\alpha)|}
\end{bmatrix}\\
& = U\cdot V
\end{split}
\end{equation*}
where $U=\begin{bmatrix}
|E_\alpha ((a+ib)t_1^\alpha)| & 0\\
& \\
0 & |E_\alpha ((a+ib)t_1^\alpha)| 
\end{bmatrix}$ is scaling matrix and $V=\begin{bmatrix}
\frac{Re[E_\alpha ((a+ib)t_1^\alpha)]}{|E_\alpha ((a+ib)t_1^\alpha)|} & \frac{Im[E_\alpha ((a+ib)t_1^\alpha)]}{|E_\alpha ((a+ib)t_1^\alpha)|}\\
& \\
-\frac{Im[E_\alpha ((a+ib)t_1^\alpha)]}{|E_\alpha ((a+ib)t_1^\alpha)|} & \frac{Re[E_\alpha ((a+ib)t_1^\alpha)]}{|E_\alpha ((a+ib)t_1^\alpha)|}
\end{bmatrix}$ is rotation matrix.
\end{proof}
\noindent {\bf Comment:}  Unlike in integer order case, the scaling not only depends on $a$ but also on $b$. The curves $X(t)$ and $U^{-1}Y(t)$ are congruent.
\begin{Ex}\label{Ex 3,4.2}
General solution of,
 \begin{equation}
{}_0^C\mathrm{D}_t^\alpha X(t)=AX(t),\,0<\alpha<1, \,\,where \,\,A=\begin{bmatrix}
-1 & 0\\
1 & -2
\end{bmatrix}\,\, and \,\, X(0)=\begin{bmatrix}
c_1\\c_2
\end{bmatrix}\label{3.19}
\end{equation}
is given by 
\begin{equation}
X(t)=\begin{bmatrix}
c_1E_\alpha(-t^\alpha)\\
c_1E_\alpha(-t^\alpha)+(c_2-c_1)E_\alpha(-2t^\alpha)
\end{bmatrix}.\label{3.20}
\end{equation} 
Let $Y(t)$ be a solution of\, ${}_0^C\mathrm{D}_t^\alpha Y(t)=AY(t)$ with $Y(0)=X(t_1)$, $t_1>0$.\\
We sketch the solution trajectories $X(t)$ (Blue color) of the system (\ref{3.19}) subject to the initial condition $X(0)=[1, 2]^T$ and $Y(t)$ (Red color) with initial condition $Y(0)=X(1)$ in the Figure \ref{Fig 7} (a).
\end{Ex}
\begin{Ex}\label{Ex 3,4.3}
General solution of,
 \begin{equation}
{}_0^C\mathrm{D}_t^\alpha X(t)=AX(t),\,0<\alpha<1, \,\,where \,\,A=\begin{bmatrix}
0 & 1\\
-4 & 0
\end{bmatrix}\,\, and \,\, X(0)=\begin{bmatrix}
c_1\\c_2
\end{bmatrix}\label{3.21}
\end{equation}
is given by 
\begin{equation}
X(t)=\begin{bmatrix}
c_1Re[E_\alpha(2it^\alpha)]+(c_2/2)Im[E_\alpha(2it^\alpha)]\\
-2c_1Im[E_\alpha(2it^\alpha)]+c_2Re[E_\alpha(2it^\alpha)]
\end{bmatrix}.\label{3.22}
\end{equation} 
Let $Y(t)$ be a solution of\, ${}_0^C\mathrm{D}_t^\alpha Y(t)=AY(t)$ with $Y(0)=X(t_1)$, $t_1>0$.\\
We sketch the solution trajectories $X(t)$ (Blue color) of the system (\ref{3.21}) with $X(0)=[1, 1]^T$ and $Y(t)$ (Red color) with $Y(0)=X(1)$ in the Figure \ref{Fig 7} (b).
\end{Ex}
 \begin{figure}[h]
               \subfloat[Solution trajectories of system (\ref{3.19})]{\includegraphics[width=0.3\textwidth]{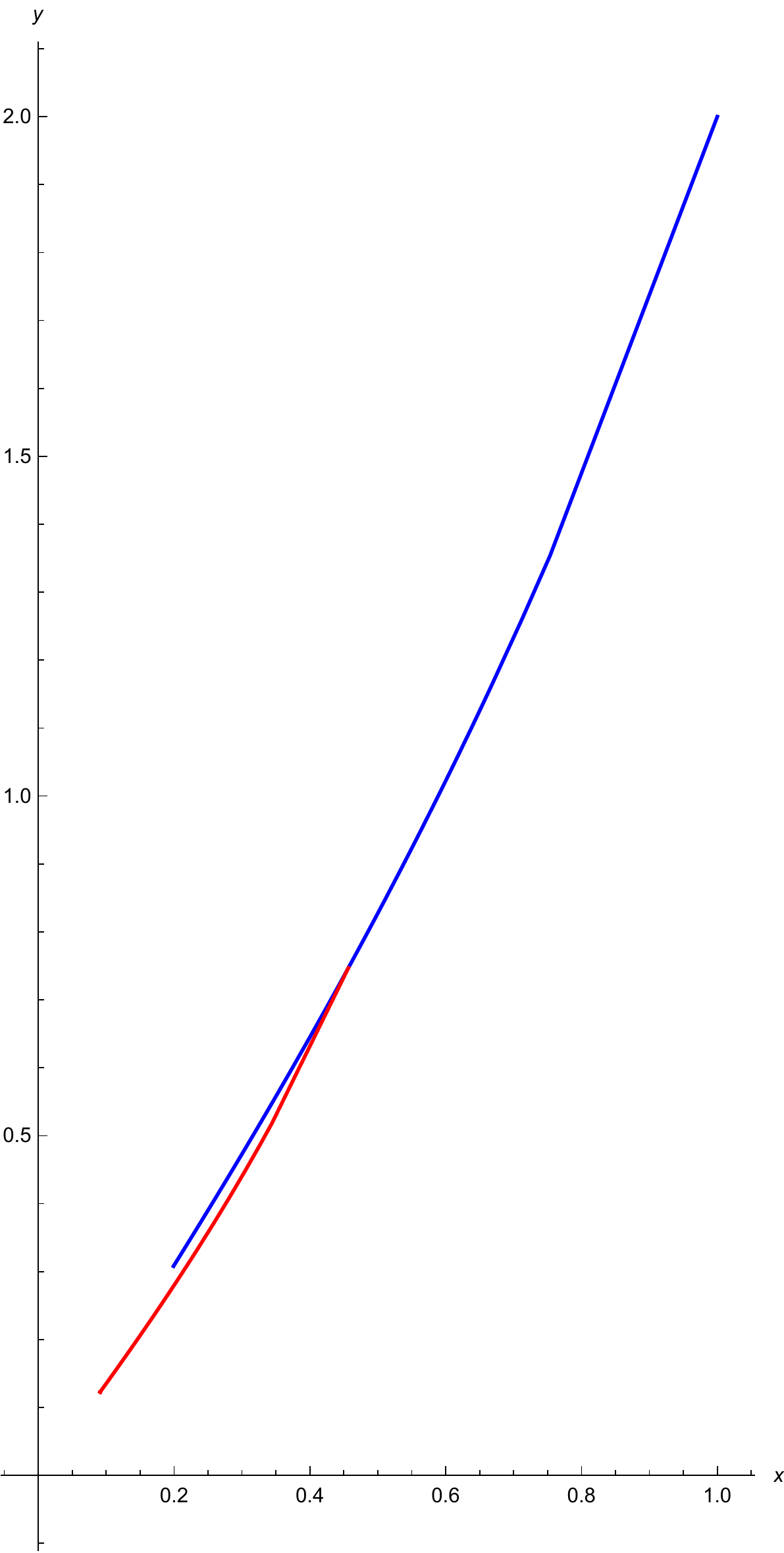}}
               \hfill 
               \subfloat[Solution trajectories of system (\ref{3.21})]{\includegraphics[width=0.4\textwidth]{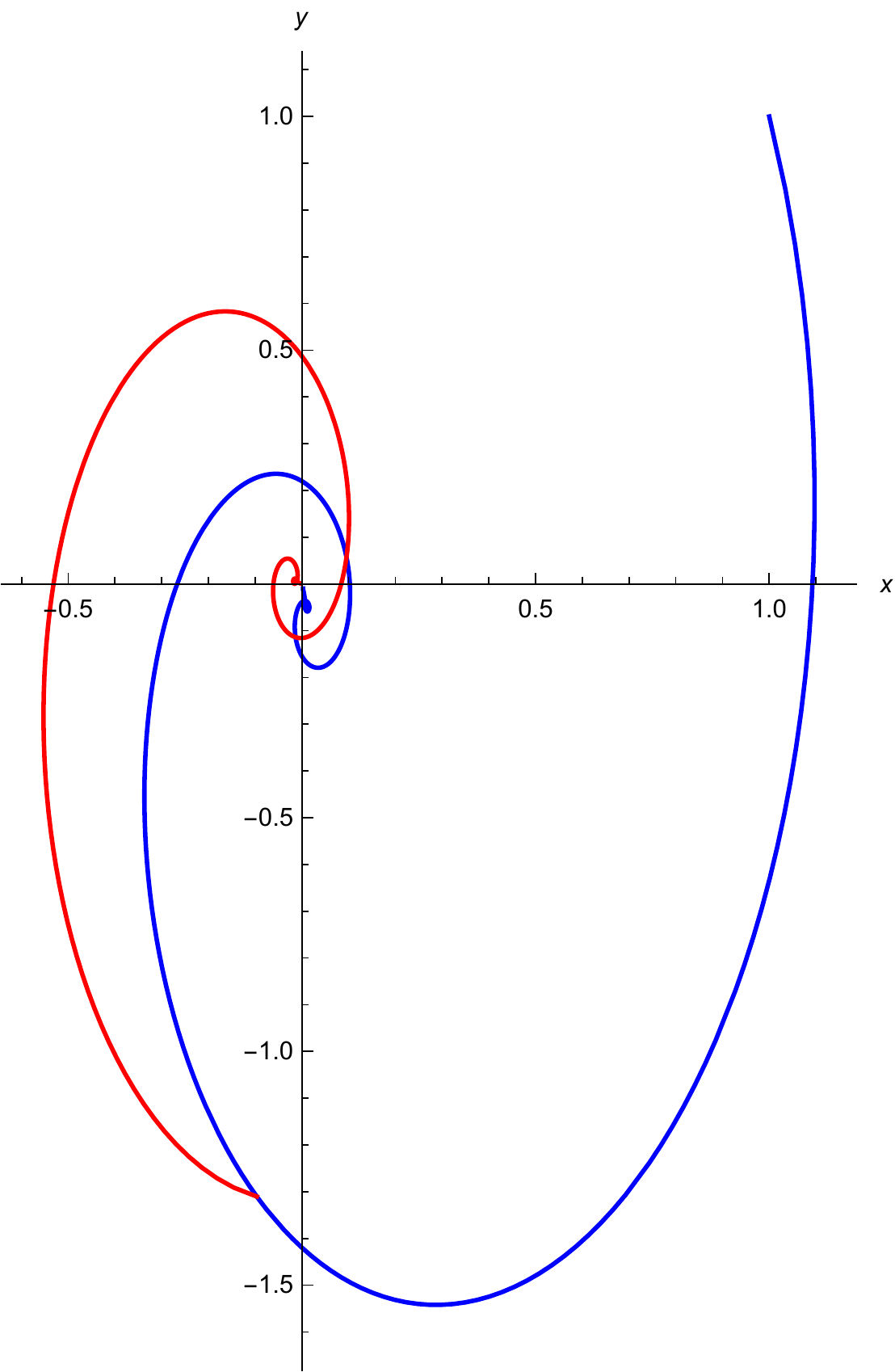}}
               \caption{
               } \label{Fig 7}
             \end{figure}      
\begin{The}\label{Thm 3,4.1}
Consider a system $\dot{X}(t)=AX(t)$, where $A$ is $3 \times 3$ matrix.\\
New trajectory $Y(t)$ starting at some point $X(t_1)$ on original trajectory $X(t)$ is given by the linear transformation
\begin{equation}
Y(t)=TX(t),\label{3.23}
\end{equation}
where $T=e^{At_1}$.\\
(i) If $A$ has real-distinct eigenvalues then $T$ represents scaling (only).\\
(ii) If $A$ has complex conjugate eigenvalues $a\pm ib$ and a real eigenvalue $\lambda$ then $T$ represents both scaling and rotation.
\end{The}
\begin{proof}

(i) If $A$ is in the standard canonical form 
$A=
\begin{bmatrix}
\lambda_1 & 0 & 0\\
0 & \lambda_2 & 0\\
0 & 0 & \lambda_3
\end{bmatrix}
$, then 
\begin{equation*}
T= e^{At_1}  = \begin{bmatrix}
e^{\lambda_1t_1} & 0 & 0\\
0 & e^{\lambda_2t_1} & 0\\
0 & 0 &  e^{\lambda_3t_1}
\end{bmatrix}
\end{equation*}
Here $T$ represents scaling only.
\\(ii) If $A$ is in the standard canonical form 
$A=
\begin{bmatrix}
a & b & 0\\
-b & a & 0\\
0 & 0& \lambda
\end{bmatrix}
$, then A has eigenvalues $a\pm i b$, $\lambda$  and 
\begin{equation*}
\begin{split}
T= e^{At_1} & = \begin{bmatrix}
e^{at_1}\cos(bt_1) & e^{at_1}\sin(bt_1) & 0\\
-e^{at_1}\sin(bt_1) & e^{at_1}\cos(bt_1) & 0\\
0 & 0 & e^{\lambda t_1}
\end{bmatrix}\\[0.05cm]
& = \begin{bmatrix}
e^{at_1} & 0 & 0\\
0 & e^{at_1} & 0\\
0 & 0 & e^{\lambda t_1}
\end{bmatrix}
\begin{bmatrix}
\cos(bt_1) & \sin(bt_1) & 0\\
-\sin(bt_1) & \cos(bt_1) & 0\\
0 & 0 & 1
\end{bmatrix}\\
& = U\cdot V
\end{split}
\end{equation*}
where $U=\begin{bmatrix}
e^{at_1} & 0 & 0\\
0 & e^{at_1} & 0\\
0 & 0 & e^{\lambda t_1}
\end{bmatrix}$ is scaling matrix (Uniform scaling by factor $e^{at_1}$ of $X$, $Y$- coordinates and  scaling of Z-coordinate by $e^{\lambda t_1}$) and $V=\begin{bmatrix}
\cos(bt_1) & \sin(bt_1) & 0\\
-\sin(bt_1) & \cos(bt_1) & 0\\
0 & 0 & 1
\end{bmatrix}$ is rotation matrix (Rotation about $Z$-axis; angle of rotation is $bt_1$).\\
The curves $X(t)$ and $U^{-1}Y(t)$ are congruent.
\end{proof}
\begin{Ex}  \label{Ex 3,4.4}
Consider the two classical systems
 \begin{equation}
\dot{X}(t)=\begin{bmatrix}
1 & 2 & -1\\
0 & 3 & -2\\
0 & 2 & -2
\end{bmatrix}X(t)\label{3.24}
\end{equation}
and
 \begin{equation}
\dot{X}(t)=\begin{bmatrix}
0 & 2 & 0\\
-2 & 0 & 0\\
0 & 0 & -3
\end{bmatrix}X(t)\label{3.25}
\end{equation}
\end{Ex}
In Figure \ref{Fig 8} (a) and \ref{Fig 8} (b) we sketch the solutions of system (\ref{3.24}) and (\ref{3.25}) respectively with initial conditions $X(0)=[
1, 1, 1
]^T$ (Blue color) and $Y(0)=X(1)$ (Red color). It can be checked that both the trajectories follow the same path.
\begin{figure}[h]
               \subfloat[Solution trajectories of system (\ref{3.24})]{\includegraphics[width=0.45\textwidth]{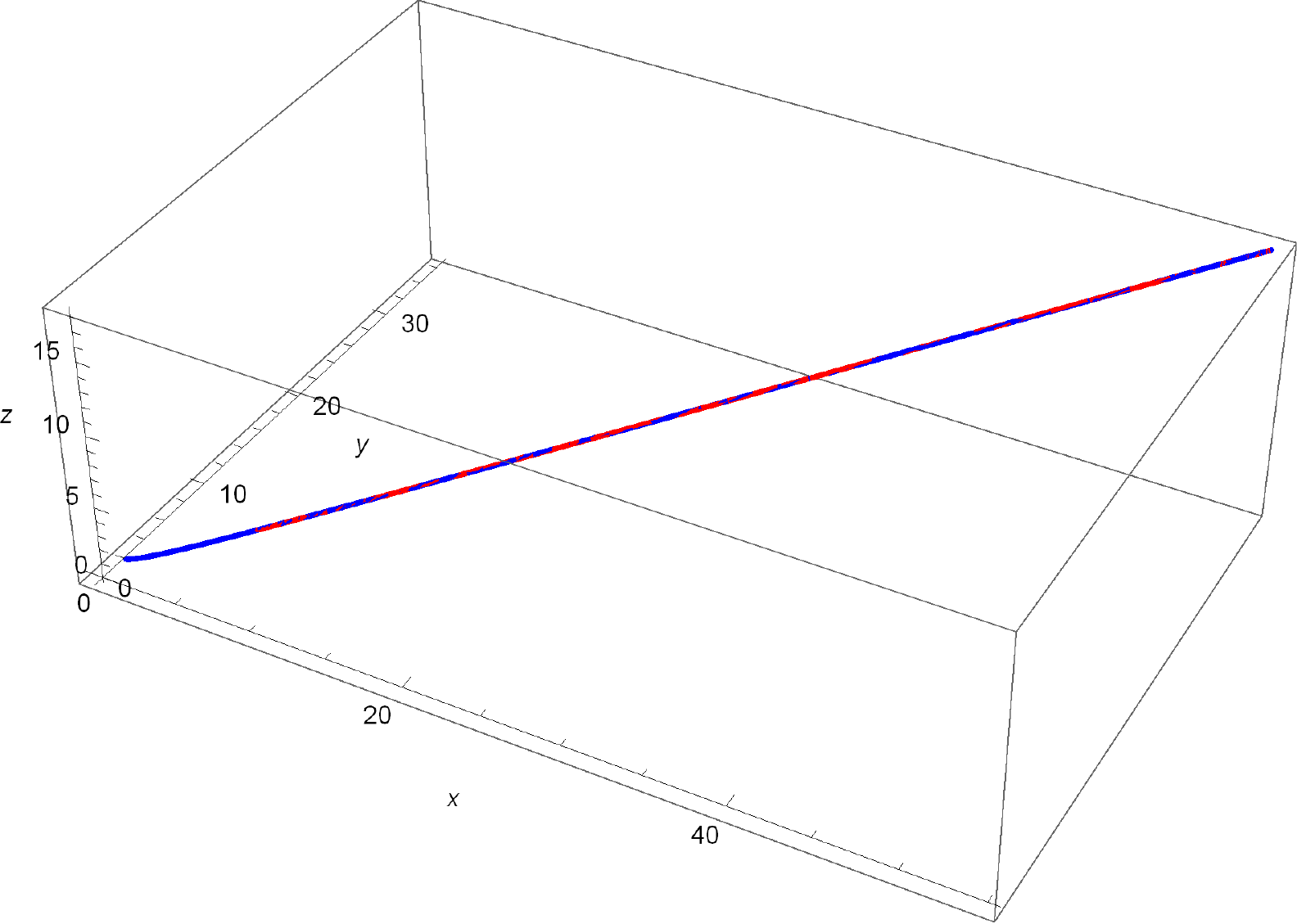}}
               \hfill 
               \subfloat[Solution trajectories of system (\ref{3.25})]{\includegraphics[width=0.4\textwidth]{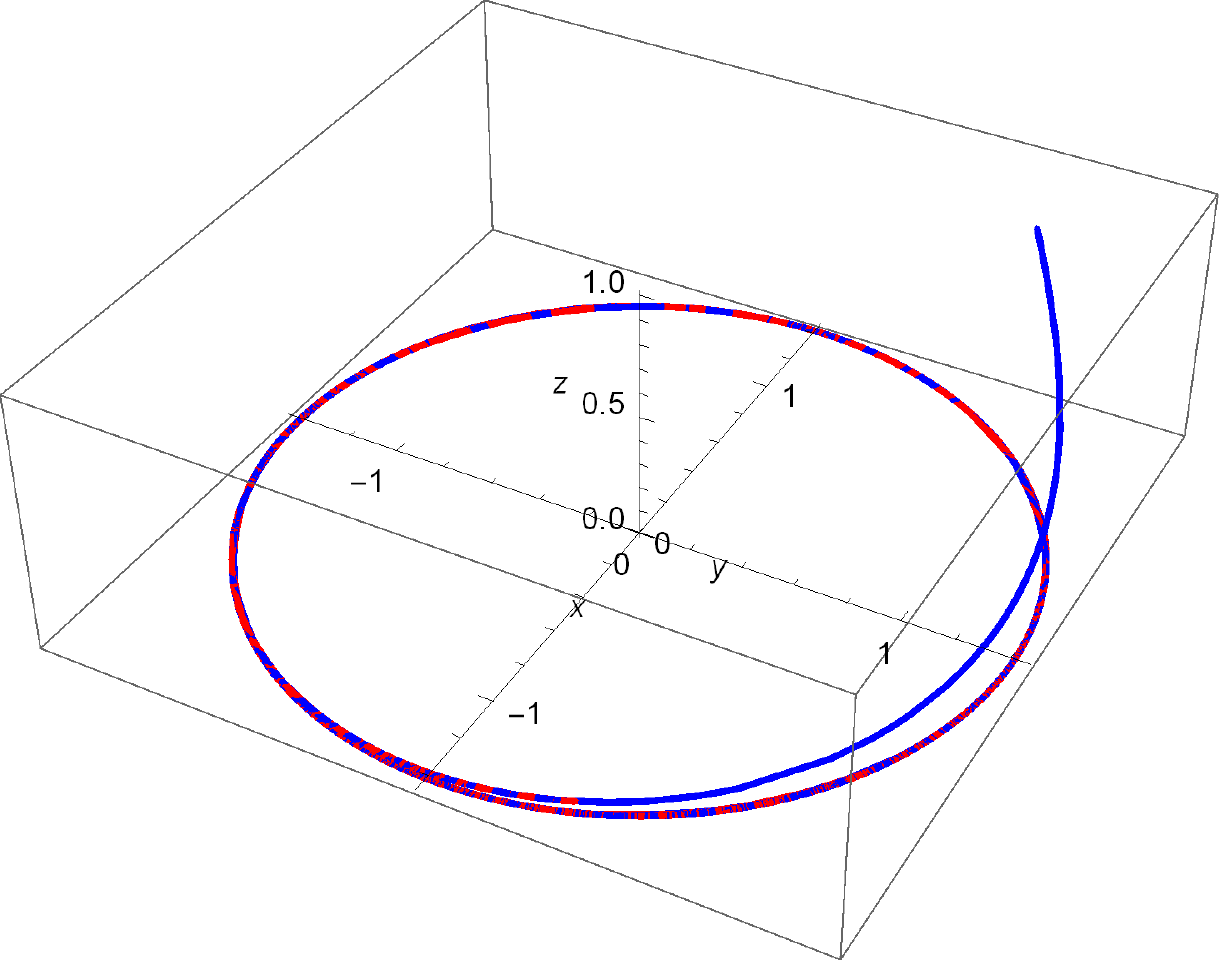}}
               \caption{
               } \label{Fig 8}
             \end{figure}  
\begin{The}\label{Thm 3,4.2}
Consider a system ${}_0^C\mathrm{D}_t^\alpha X(t)=AX(t)$,\,$0<\alpha<1$\, where $A$ is $3 \times 3$ matrix.\\
New trajectory $Y(t)$ starting at some point $X(t_1)$ on original trajectory $X(t)$ is given by the linear transformation
\begin{equation}
Y(t)=TX(t),\label{3.26}
\end{equation}
where $T=E_\alpha(At_1^\alpha)$.\\
(i) If $A$ has real-distinct eigenvalues then $T$ represents scaling (only).\\
(ii) If $A$ has complex conjugate eigenvalues $a\pm ib$ and a real eigenvalue $\lambda$ then $T$ represents both scaling and rotation.
\end{The}
\begin{proof}
 
(i) If $A$ is in the standard canonical form 
$A=
\begin{bmatrix}
\lambda_1 & 0 & 0\\
0 & \lambda_2 & 0\\
0& 0& \lambda_3
\end{bmatrix}
$, then 
\begin{equation*}
T= E_\alpha(At_1^\alpha)  = \begin{bmatrix}
E_\alpha(\lambda_1 t_1^\alpha) & 0 & 0\\
0 & E_\alpha(\lambda_2 t_1^\alpha) & 0\\
0 & 0 & E_\alpha(\lambda_3 t_1^\alpha)
\end{bmatrix}
\end{equation*}
Here $T$ represents scaling only.
\\(ii) If $A$ is in the standard canonical form 
$A=
\begin{bmatrix}
a & b & 0\\
-b & a & 0\\
0 & 0 & \lambda
\end{bmatrix}
$ then 
\begin{equation*}
\begin{split}
T= e^{At_1} & = \begin{bmatrix}
Re[E_\alpha ((a+ib)t_1^\alpha)] & Im[E_\alpha ((a+ib)t_1^\alpha)] & 0\\
-Im[E_\alpha ((a+ib)t_1^\alpha)] & Re[E_\alpha ((a+ib)t_1^\alpha)] & 0\\
0 & 0 & E_\alpha(\lambda t_1^\alpha)
\end{bmatrix}\\[0.05cm]
& = \begin{bmatrix}
|E_\alpha ((a+ib)t_1^\alpha)| & 0 & 0\\
0 & |E_\alpha ((a+ib)t_1^\alpha)| & 0\\
0 & 0 & E_\alpha(\lambda t_1^\alpha)
\end{bmatrix}
\begin{bmatrix}
\frac{Re[E_\alpha ((a+ib)t_1^\alpha)]}{|E_\alpha ((a+ib)t_1^\alpha)|} & \frac{Im[E_\alpha ((a+ib)t_1^\alpha)]}{|E_\alpha ((a+ib)t_1^\alpha)|} & 0\\
-\frac{Im[E_\alpha ((a+ib)t_1^\alpha)]}{|E_\alpha ((a+ib)t_1^\alpha)|} & \frac{Re[E_\alpha ((a+ib)t_1^\alpha)]}{|E_\alpha ((a+ib)t_1^\alpha)|} & 0\\
0 & 0 & 1
\end{bmatrix} \\
& = U\cdot V
\end{split}
\end{equation*}
where $U=\begin{bmatrix}
|E_\alpha ((a+ib)t_1^\alpha)| & 0 & 0\\
0 & |E_\alpha ((a+ib)t_1^\alpha)| & 0\\
0 & 0 & E_\alpha(\lambda t_1^\alpha)
\end{bmatrix}$ is scaling matrix (Uniform scaling by factor $|E_\alpha ((a+ib)t_1^\alpha)|$ of $X$,  $Y$- coordinates and scaling of Z-coordinate by $E_\alpha(\lambda t_1^\alpha)$) \\and $V=\begin{bmatrix}
\frac{Re[E_\alpha ((a+ib)t_1^\alpha)]}{|E_\alpha ((a+ib)t_1^\alpha)|} & \frac{Im[E_\alpha ((a+ib)t_1^\alpha)]}{|E_\alpha ((a+ib)t_1^\alpha)|} & 0\\
-\frac{Im[E_\alpha ((a+ib)t_1^\alpha)]}{|E_\alpha ((a+ib)t_1^\alpha)|} & \frac{Re[E_\alpha ((a+ib)t_1^\alpha)]}{|E_\alpha ((a+ib)t_1^\alpha)|} & 0\\
0 & 0 & 1
\end{bmatrix} $ is rotation matrix (Rotation about $Z$-axis; angle of rotation is $\theta=\mathrm{cos}^{-1}\left[\frac{Re[E_\alpha ((a+ib)t_1^\alpha)]}{|E_\alpha ((a+ib)t_1^\alpha)|}\right]$).\\
The curves $X(t)$ and $U^{-1}Y(t)$ are congruent.
\end{proof}
{\bf Comments}:- \\
New trajectories are transformed versions of original trajectories. In integer order case, both trajectories follow same path because $e^{A(t_1+t_2)}=e^{At_1}e^{At_2}$.
\par This is not the case with fractional order systems because $E_\alpha (A(t_1+t_2)^\alpha)\ne E_\alpha (At_1^\alpha)E_\alpha (At_2^\alpha)$, in general.

\begin{Ex}\label{Ex 3,4.5}
General solution of,
 \begin{equation}
{}_0^C\mathrm{D}_t^\alpha X(t)=AX(t),\,0<\alpha<1, \,\,where \,\,A=\begin{bmatrix}
1 & 2 & -1\\
0 & 3 & -2\\
0 & 2 & -2
\end{bmatrix}\,\, and \,\, X(0)=\begin{bmatrix}
c_1\\c_2\\c_3
\end{bmatrix}\label{3.27}
\end{equation}
is given by 
\begin{equation}
X(t)=\begin{bmatrix}
c_1E_\alpha(t^\alpha)+2c_2(E_\alpha(2t^\alpha)-E_\alpha(t^\alpha))+c_3(E_\alpha(t^\alpha)-E_\alpha(2t^\alpha))\\
(c_2/3)(4E_\alpha(2t^\alpha)-E_\alpha(-t^\alpha))+(2c_3/3)(E_\alpha(-t^\alpha)-E_\alpha(2t^\alpha))\\
(2c_2/3)(E_\alpha(2t^\alpha)-E_\alpha(-t^\alpha))+(c_3/3)(4E_\alpha(-t^\alpha)-E_\alpha(2t^\alpha))
\end{bmatrix}.\label{3.28}
\end{equation} 
$Y(t)$ be a solution of\, ${}_0^C\mathrm{D}_t^\alpha Y(t)=AY(t)$ with $Y(0)=X(t_1)$, $t_1>0$.\\
In the Figure \ref{Fig 9} (a), we sketch the solution trajectories $X(t)$ (Blue color) of the system (\ref{3.27}) subject to the initial condition $X(0)=[1, 1, 1]^T$ and $Y(t)$ (Red color) with initial condition $Y(0)=X(1)$. 
\end{Ex}
\begin{Ex} \label{Ex 3,4.6}
Repeating the same exercise as in Example \ref{Ex 3,4.5}, with 
 \begin{equation}
A=\begin{bmatrix}
-3 & 0 & 0\\
0 & 3 & -2\\
0 & 1 & 1
\end{bmatrix}\,\, and \,\, X(0)=\begin{bmatrix}
c_1\\c_2\\c_3
\end{bmatrix}\label{3.29}
\end{equation}
we get 
\begin{equation}
X(t)=\begin{bmatrix}
c_1E_\alpha(-3t^\alpha)\\
c_2Re[E_\alpha((2+i)t^\alpha)]+(c_2-2c_3)Im[E_\alpha((2+i)t^\alpha)]\\
c_3Re[E_\alpha((2+i)t^\alpha)]+(c_2-c_3)Im[E_\alpha((2+i)t^\alpha)]
\end{bmatrix}.\label{3.30}
\end{equation} 
In the Figure \ref{Fig 9} (b), we sketch the solution trajectories $X(t)$ (Blue color) of the system (\ref{3.29}) with initial condition $X(0)=[1, 1, 1]^T$ and $Y(t)$ (Red color) subject to the initial condition $Y(0)=X(1)$. 
\end{Ex}
\begin{figure}[h]
               \subfloat[]{\includegraphics[width=0.45\textwidth]{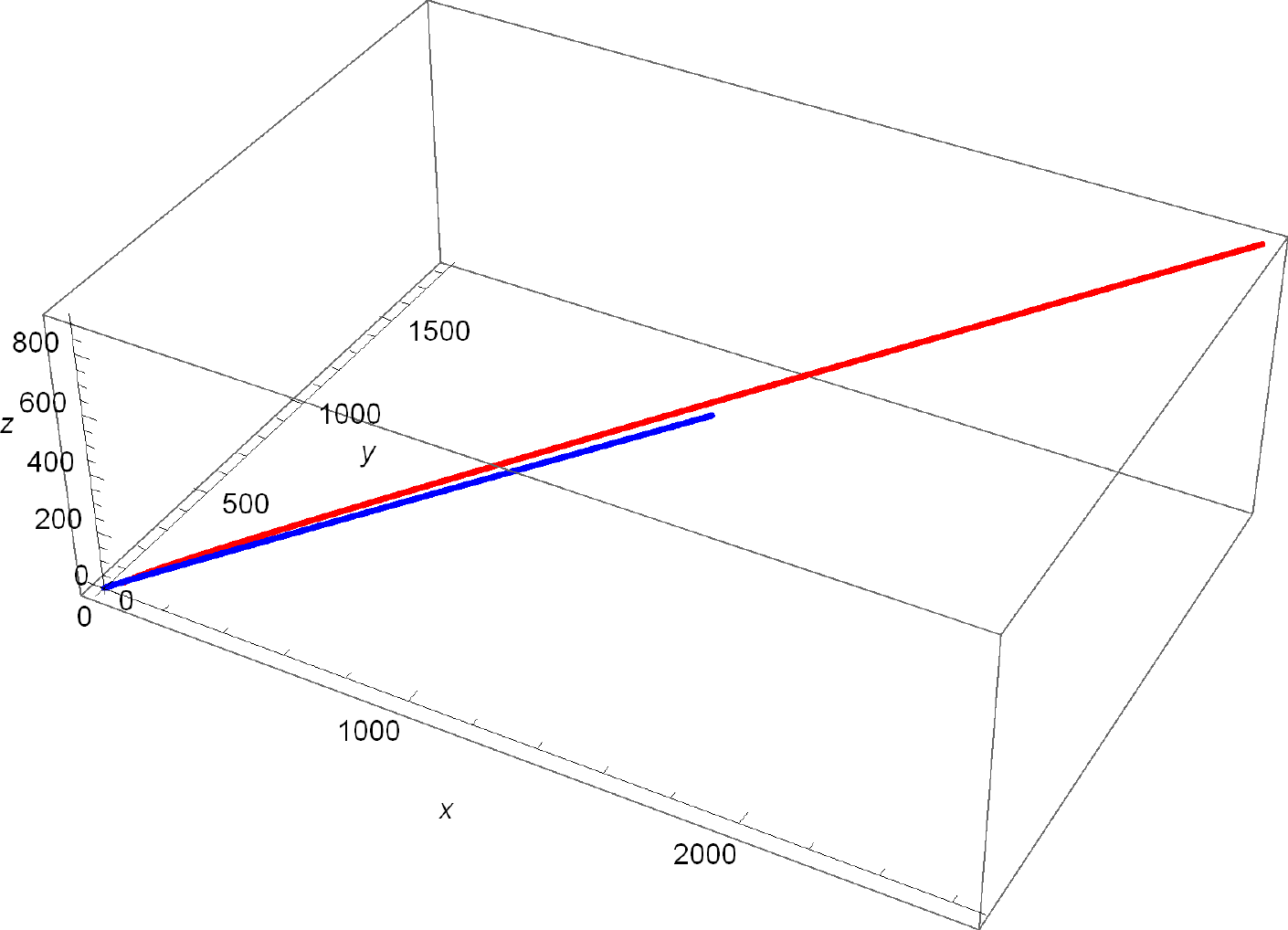}}
               \hfill 
               \subfloat[]{\includegraphics[width=0.3\textwidth]{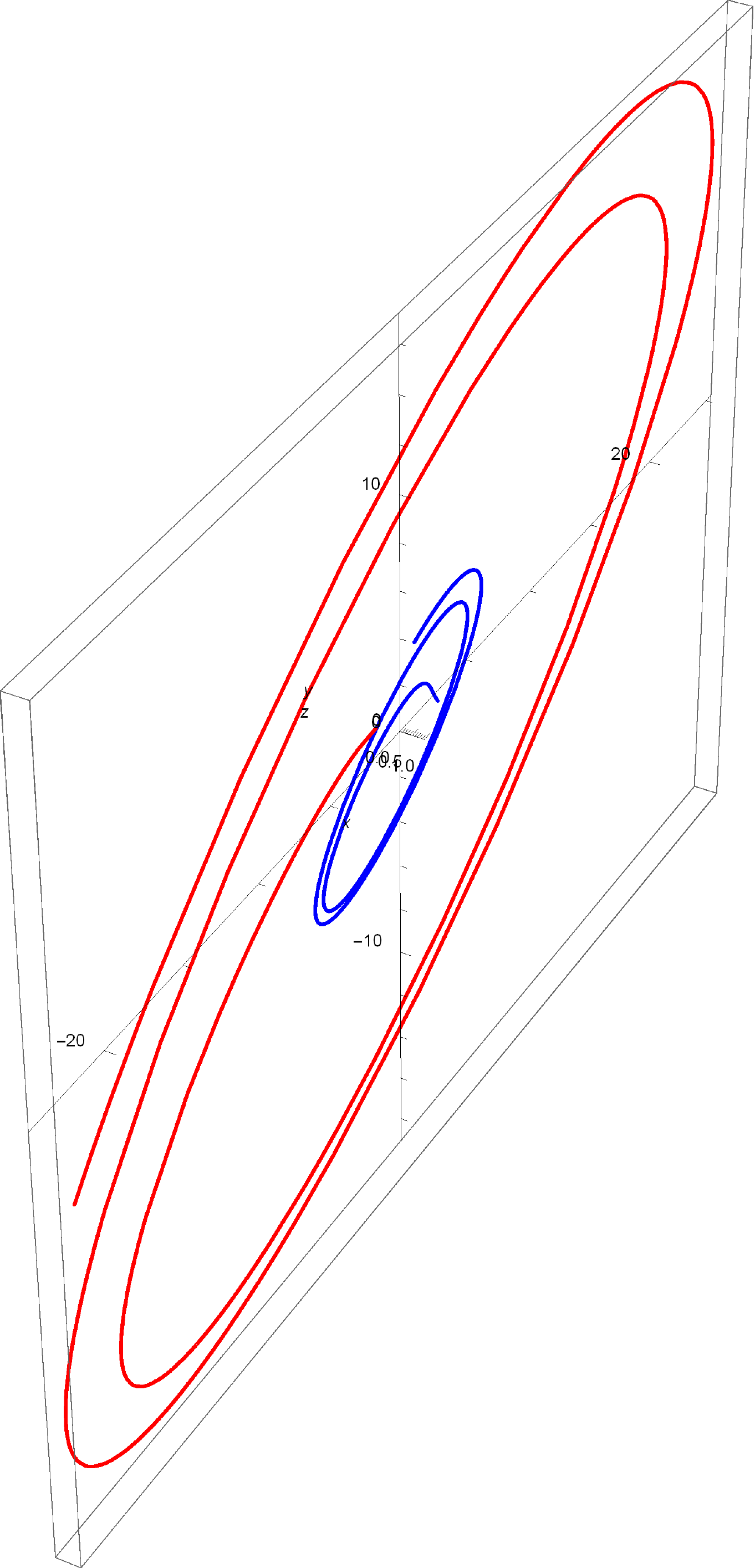}}
               \caption{Trajectories of fractional order systems (\ref{3.27}) and (\ref{3.29}).
               } \label{Fig 9}
             \end{figure}  
\section{Differential geometry of trajectories of fractional order systems}
Frenet apparatus is a tool which is very useful to describe the shape of a curve.
 In this section we find Frenet apparatus for solution trajectories of FDEs 
  \begin{equation}
{}_0^C\mathrm{D}_t^\alpha X(t)=AX(t),\,0<\alpha<1,\, \mathrm{with}\,\, X(0)=X_0=\begin{bmatrix}
 c_1\\
 c_2
 \end{bmatrix},\label{3.31}
 \end{equation}
 where $A$ is in canonical form.\\
 {\bf(1)} Let, $A= \begin{bmatrix}
 \lambda_1 & 0\\
 0 & \lambda_2
 \end{bmatrix}$, where $\lambda_1\ne\lambda_2$ are real numbers.\\
 The Frenet apparatus of solution trajectory of 
 ${}_0^C\mathrm{D}_t^\alpha X(t)=AX(t),\,0<\alpha<1,\, X(0)=(
  c_1, 
  c_2
  )^T$ (respectively ${}_0^C\mathrm{D}_t^\alpha Y(t)=AY(t),\,0<\alpha<1,\, Y(0)=X(t_1),\, t_1>0$) is $T_1, N_1, \kappa_1$ (respectively $T_2, N_2, \kappa_2$).\\
  If $ \nu_1 (t)=\sqrt{(\dot{x})^2+(\dot{y})^2}$, is speed of $X=(x, y)^T$ then
  $$\nu_1(t)= \sqrt{ c_1^2\left(\lambda_1t^{\alpha-1}E_{\alpha, \alpha}(\lambda_1t^\alpha)\right)^2+c_2^2\left(\lambda_2t^{\alpha-1}E_{\alpha, \alpha}(\lambda_2t^\alpha)\right)^2}$$
  Similarly, speed of $Y$ is given by
   $$\nu_2(t)= \sqrt{ c_1^2(E_\alpha(\lambda_1t_1^\alpha))^2\left(\lambda_1t^{\alpha-1}E_{\alpha, \alpha}(\lambda_1t^\alpha)\right)^2+c_2^2(E_\alpha(\lambda_2t_1^\alpha))^2\left(\lambda_2t^{\alpha-1}E_{\alpha, \alpha}(\lambda_2t^\alpha)\right)^2}.$$
   If $u_1= \lambda_1t^{\alpha-1}E_{\alpha, \alpha}(\lambda_1t^\alpha)\left[\frac{\mathrm{d}}{\mathrm{dt}}(\lambda_2t^{\alpha-1}E_{\alpha, \alpha}(\lambda_2t^\alpha))\right]-\lambda_2t^{\alpha-1}E_{\alpha, \alpha}(\lambda_2t^\alpha)\left[\frac{\mathrm{d}}{\mathrm{dt}}(\lambda_1t^{\alpha-1}E_{\alpha, \alpha}(\lambda_1t^\alpha))\right]$ then we have 
  \begin{equation*}
  \kappa_1=\left|\frac{c_1 c_2u_1}{\nu_1^3}\right|\,\, \mathrm{and}\,\, \kappa_2=\left|\frac{c_1 c_2E_{ \alpha}(\lambda_1t_1^\alpha)E_{ \alpha}(\lambda_2t_1^\alpha)u_1}{\nu_2^3}\right|.
  \end{equation*}
  Therefore, $$\kappa_2=\left|\frac{E_{ \alpha}(\lambda_1t_1^\alpha)E_{ \alpha}(\lambda_2t_1^\alpha)\nu_1^3}{\nu_2^3}\right|\kappa_1.$$
The unit tangent vectors
  \begin{equation*}
    T_1=\frac{\left(c_1\lambda_1t^{\alpha-1}E_{\alpha, \alpha}(\lambda_1t^\alpha),c_2\lambda_2t^{\alpha-1}E_{\alpha, \alpha}(\lambda_2t^\alpha)\right)}{\nu_1 (t)}
    \end{equation*}
and 
\begin{equation*}
	T_2=\frac{\left(c_1E_\alpha(\lambda_1t_1^\alpha)\lambda_1t^{\alpha-1}E_{\alpha, \alpha}(\lambda_1t^\alpha),c_2E_\alpha(\lambda_2t_1^\alpha)\lambda_2t^{\alpha-1}E_{\alpha, \alpha}(\lambda_2t^\alpha)\right)}{\nu_2 (t)}.
\end{equation*}
$$
\therefore T_2=\frac{\nu_1 (t)}{\nu_2 (t)}\begin{bmatrix}
E_\alpha(\lambda_1t_1^\alpha) & 0 \\
0 & E_\alpha(\lambda_2t_1^\alpha)
\end{bmatrix}T_1.
$$
Similarly, the unit normal vectors
\begin{equation*}
N_1=\frac{\left(-c_2\lambda_2t^{\alpha-1}E_{\alpha, \alpha}(\lambda_2t^\alpha),c_1\lambda_1t^{\alpha-1}E_{\alpha, \alpha}(\lambda_1t^\alpha)\right)}{\nu_1 (t)}
\end{equation*}
and 
\begin{equation*}
N_2=\frac{\left(-c_2E_\alpha(\lambda_2t_1^\alpha)\lambda_2t^{\alpha-1}E_{\alpha, \alpha}(\lambda_2t^\alpha),c_1E_\alpha(\lambda_1t_1^\alpha)\lambda_1t^{\alpha-1}E_{\alpha, \alpha}(\lambda_1t^\alpha)\right)}{\nu_2 (t)}.
\end{equation*}
$$
\therefore N_2=\frac{\nu_1 (t)}{\nu_2 (t)}\begin{bmatrix}
E_\alpha(\lambda_2t_1^\alpha) & 0 \\
0 & E_\alpha(\lambda_1t_1^\alpha)
\end{bmatrix}N_1.
$$
Note that, if $\lambda_1=\lambda_2=\lambda$ then\,
$
\nu_2(t)=E_\alpha(\lambda t_1^\alpha)\nu_1(t),\,T_2= T_1, \, N_2= N_1, \, \kappa_2= \kappa_1=0.
$\\[0.2cm]   
{\bf Conclusions:}\\
 {\bf(2)} Let, $A= \begin{bmatrix}
  \lambda & 1\\
  0 & \lambda
  \end{bmatrix}$. \\
  (i) In this case, the general solution of the system is, $X(t)=\begin{bmatrix}
  c_1 E_\alpha(\lambda t^\alpha)+c_2\frac{t^\alpha}{\alpha}E_{\alpha,\alpha}(\lambda t^\alpha)\\
  c_2 E_\alpha(\lambda t^\alpha)
  \end{bmatrix}$.\\
 \begin{equation*}
       \begin{split}
      \nu_1 (t)& = \left( (c_1^2+c_2^2)\left(\sum_{n=1}^{\infty}\frac{\lambda^n t^{\alpha n-1}}{\Gamma(\alpha n)}\right)^2+ \frac{c_2^2}{\alpha^2}\left(\sum_{n=1}^{\infty}\frac{\lambda^n (\alpha n+\alpha) t^{\alpha n+\alpha-1}}{\Gamma(\alpha n+\alpha)}\right)^2\right.\\
      & \quad \left.+\frac{2c_1c_2}{\alpha}\sum_{n=1}^{\infty}\frac{\lambda^n t^{\alpha n-1}}{\Gamma(\alpha n)}\sum_{n=1}^{\infty}\frac{\lambda^n (\alpha n+\alpha) t^{\alpha n+\alpha-1}}{\Gamma(\alpha n+\alpha)}\right)^{1/2}.
     \end{split}
       \end{equation*} 
       If $$u_2(t)=\sum_{n=1}^{\infty}\frac{\lambda^n (\alpha n+\alpha)t^{\alpha n+\alpha-1}}{\Gamma(\alpha n+\alpha)}\sum_{n=1}^{\infty}\frac{\lambda^n  t^{\alpha n-2}}{\Gamma(\alpha n-1)}-\sum_{n=1}^{\infty}\frac{\lambda^n (\alpha n+\alpha)t^{\alpha n+\alpha-2}}{\Gamma(\alpha n+\alpha-1)}\sum_{n=1}^{\infty}\frac{\lambda^n t^{\alpha n-1}}{\Gamma(\alpha n)}, $$ then 
   \begin{equation*}
   \kappa_1  = \left|\frac{ \frac{c_2^2}{\alpha}u_2(t)}
   {(\nu_1 (t))^3}\right|,
   \end{equation*}
   \begin{equation*}
     T_1=\frac{\left(c_1\sum_{n=1}^{\infty}\frac{\lambda^n t^{\alpha n-1}}{\Gamma(\alpha n)}+\frac{c_2}{\alpha}\sum_{n=1}^{\infty}\frac{\lambda^n(\alpha n +\alpha)t^{\alpha n+\alpha-1}}{\Gamma(\alpha n +\alpha)},c_2\sum_{n=1}^{\infty}\frac{\lambda^n t^{\alpha n-1}}{\Gamma(\alpha n)}\right)}{\nu_1 (t)}
     \end{equation*}
     and
  \begin{equation*}
      N_1=\frac{\left(-c_2\sum_{n=1}^{\infty}\frac{\lambda^n t^{\alpha n-1}}{\Gamma(\alpha n)},c_1\sum_{n=1}^{\infty}\frac{\lambda^n t^{\alpha n-1}}{\Gamma(\alpha n)}+\frac{c_2}{\alpha}\sum_{n=1}^{\infty}\frac{\lambda^n(\alpha n +\alpha)t^{\alpha n+\alpha-1}}{\Gamma(\alpha n +\alpha)}\right)}{\nu_1 (t)}.
      \end{equation*} 
      
   Similarly,    $Y(t)=\begin{bmatrix}
    (c_1 E_\alpha(\lambda t_1^\alpha)+c_2\frac{t_1^\alpha}{\alpha}E_{\alpha,\alpha}(\lambda t_1^\alpha)) E_\alpha(\lambda t^\alpha)+c_2\frac{t^\alpha}{\alpha}E_\alpha(\lambda t_1^\alpha)E_{\alpha,\alpha}(\lambda t^\alpha)\\
    c_2 E_\alpha(\lambda t_1^\alpha) E_\alpha(\lambda t^\alpha)
    \end{bmatrix}$ and\\
 \begin{equation*}
        \begin{split}
              \nu_2 (t)& = \left((E_\alpha(\lambda t_1^\alpha)^2 \left[(c_1^2+c_2^2)\left(\sum_{n=1}^{\infty}\frac{\lambda^n t^{\alpha n-1}}{\Gamma(\alpha n)}\right)^2+ \frac{c_2^2}{\alpha^2}\left(\sum_{n=1}^{\infty}\frac{\lambda^n (\alpha n+\alpha) t^{\alpha n+\alpha-1}}{\Gamma(\alpha n+\alpha)}\right)^2\right.\right.\\
              & \quad \left.\left.+\frac{2c_1c_2}{\alpha}\sum_{n=1}^{\infty}\frac{\lambda^n t^{\alpha n-1}}{\Gamma(\alpha n)}\sum_{n=1}^{\infty}\frac{\lambda^n (\alpha n+\alpha) t^{\alpha n+\alpha-1}}{\Gamma(\alpha n+\alpha)}\right]
              +\frac{t_1^\alpha}{\alpha}E_{\alpha,\alpha}(\lambda t_1^\alpha)E_\alpha(\lambda t_1^\alpha)\sum_{n=1}^{\infty}\frac{\lambda^n t^{\alpha n-1}}{\Gamma(\alpha n)}\right.\\
    & \quad \left.\left[2c_1c_2\sum_{n=1}^{\infty}\frac{\lambda^n t^{\alpha n-1}}{\Gamma(\alpha n)}+\frac{2c_2^2}{\alpha}\sum_{n=1}^{\infty}\frac{\lambda^n (\alpha n+\alpha) t^{\alpha n+\alpha-1}}{\Gamma(\alpha n+\alpha)} \right] +\frac{c_2^2t_1^{2\alpha}}{\alpha^2}(E_{\alpha,\alpha}(\lambda t_1^\alpha))^2\left(\sum_{n=1}^{\infty}\frac{\lambda^n t^{\alpha n-1}}{\Gamma(\alpha n)}\right)^2        
              \right)^{1/2}.
             \end{split}
         \end{equation*}      
     \begin{equation*}
 \therefore    \kappa_2  = \left|\frac{ (\nu_1 (t))^3}
        {(\nu_2 (t))^3}E_\alpha(\lambda t_1^\alpha)\right|\kappa_1,
     \end{equation*}
     \begin{equation*}
       T_2=\frac{\nu_1 (t)E_\alpha(\lambda t_1^\alpha)}{\nu_2 (t)}T_1 +\frac{t_1^\alpha E_{\alpha,\alpha}(\lambda t_1^\alpha)}{\nu_2(t)}V,\quad \mathrm{where}\,\, V=\left(\frac{c_2}{\alpha}\sum_{n=1}^{\infty}\frac{\lambda^n t^{\alpha n-1}}{\Gamma(\alpha n)}, 0\right) \,\,\mathrm{and}
       \end{equation*}
    \begin{equation*}
        N_2=\frac{\nu_1 (t)E_\alpha(\lambda t_1^\alpha)}{\nu_2 (t)}N_1 + \frac{t_1^\alpha E_{\alpha,\alpha}(\lambda t_1^\alpha)}{\nu_2(t)}W,\quad \mathrm{where}\,\,  W=\left(0, \frac{c_2}{\alpha}\sum_{n=1}^{\infty}\frac{\lambda^n t^{\alpha n-1}}{\Gamma(\alpha n)}\right).
        \end{equation*} 
        
  {\bf(4)} Let, $A= \begin{bmatrix}
   a & b\\
   -b & a
   \end{bmatrix}$. \\
   (i) In this case, the general solution of the system is,  $X(t)=\begin{bmatrix}
  c_1 Re[E_\alpha ((a+ib)t^\alpha)] + c_2Im[E_\alpha ((a+ib)t^\alpha)] \\
   -c_1Im[E_\alpha ((a+ib)t^\alpha)] + c_2Re[E_\alpha ((a+ib)t^\alpha)] 
   \end{bmatrix}$.\\
   Let, $$c_\alpha(a,b,t)=\sum_{n=1}^{\infty}\frac{(Re(a+ib)^n )t^{\alpha n-1}}{\Gamma(\alpha n)},\, s_\alpha(a,b,t)=\sum_{n=1}^{\infty}\frac{(Im(a+ib)^n )t^{\alpha n-1}}{\Gamma(\alpha n)} \,\,\mathrm{and} $$
   $$
   u_3(t)=(c_1^2+c_2^2)\left(-c_\alpha(a,b,t)\sum_{n=1}^{\infty}\frac{(Im(a+ib)^n ) (\alpha n-1) t^{\alpha n-2}}{\Gamma(\alpha n)}+s_\alpha(a,b,t)\sum_{n=1}^{\infty}\frac{(Re(a+ib)^n ) (\alpha n-1) t^{\alpha n-2}}{\Gamma(\alpha n)}\right).
   $$
 \begin{equation*}
 \therefore      \nu_1 (t)= \sqrt{c_1^2+c_2^2}\left( \left(c_\alpha(a,b,t)\right)^2+\left(s_\alpha(a,b,t)\right)^2\right)^{1/2},
        \end{equation*} 
    \begin{equation*}
    \kappa_1=\left|\frac{u_3(t)}{(\nu_1 (t))^3}\right|,
    \end{equation*}
    \begin{equation*}
      T_1=\frac{1}{\nu_1 (t)}\left(c_1c_\alpha(a,b,t)+c_2s_\alpha(a,b,t),\,-c_1s_\alpha(a,b,t)+c_2c_\alpha(a,b,t)\right) \,\,\mathrm{and}
      \end{equation*}
   \begin{equation*}
       N_1=\frac{1}{\nu_1 (t)}\left(c_1s_\alpha(a,b,t)-c_2c_\alpha(a,b,t),\, c_1c_\alpha(a,b,t)+c_2s_\alpha(a,b,t)\right).
       \end{equation*} 
    Similarly,    $$Y(t)=\begin{bmatrix}
     p Re[E_\alpha ((a+ib)t^\alpha)] + q Im[E_\alpha ((a+ib)t^\alpha)] \\
        -pIm[E_\alpha ((a+ib)t^\alpha)] + qRe[E_\alpha ((a+ib)t^\alpha)], 
     \end{bmatrix}$$
     where $p=c_1 Re[E_\alpha ((a+ib)t_1^\alpha)] + c_2Im[E_\alpha ((a+ib)t_1^\alpha)]$ and $q=-c_1Im[E_\alpha ((a+ib)t_1^\alpha)] + c_2Re[E_\alpha ((a+ib)t_1^\alpha)]$.
 \begin{equation*}
           \begin{split}
       \therefore  \nu_2 (t)& = \sqrt{c_1^2+c_2^2}\,\,|E_\alpha ((a+ib)t_1^\alpha)|\left( \left(c_\alpha(a,b,t)\right)^2+\left(s_\alpha(a,b,t)\right)^2\right)^{1/2}\\
         & = |E_\alpha ((a+ib)t_1^\alpha)|\, \nu_1(t),
        \end{split}  
          \end{equation*}  
      \begin{equation*}
      \kappa_2=\frac{1}{|E_\alpha ((a+ib)t_1^\alpha)|}\kappa_1,
      \end{equation*}
      \begin{equation*}
        T_2=\frac{Re[E_\alpha ((a+ib)t_1^\alpha)]\nu_1(t)}{\nu_2(t)}T_1 - \frac{Im[E_\alpha ((a+ib)t_1^\alpha)]\nu_1(t)}{\nu_2(t)}N_1 \,\,\mathrm{and}
        \end{equation*}
     \begin{equation*}
         N_2=\frac{Re[E_\alpha ((a+ib)t_1^\alpha)]\nu_1(t)}{\nu_2(t)}N_1 + \frac{Re[E_\alpha ((a+ib)t_1^\alpha)]\nu_1(t)}{\nu_2(t)}T_1.
         \end{equation*}  
 {\bf Comment:} Speed of the new curve is affected by scaling factor.                     

\section{Bifurcation Analysis}
 Since the new trajectory $Y(t)$ starting at some point $X(t_1)$ on original trajectory is a transformation of solution $X(t)$ of ${}_0^C\mathrm{D}_t^\alpha X(t)=\begin{bmatrix}
a & b\\
-b & a
\end{bmatrix}X(t),\, 0<\alpha<1$, it is worth studying the effect of fractional order on such transformations.\\
(i) Fix $t_1=1$, $a=0.983469$ and $b=0.181075$.\\
\begin{figure}[h]
   \begin{center}
             \includegraphics[width=0.5\textwidth]{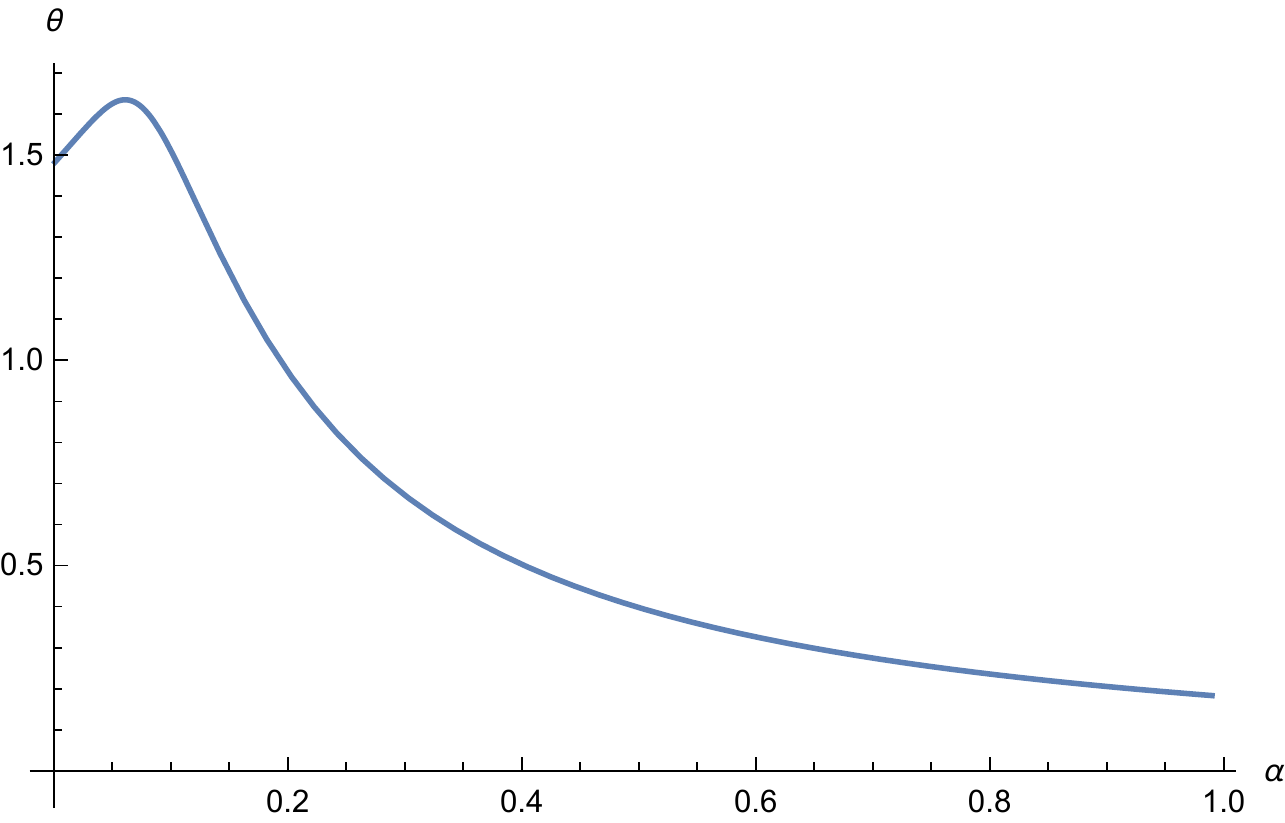}
            \caption{Angle of rotation $\theta$ versus $\alpha$. }         
   \label{Fig 6.1}
   \end{center}  
          \end{figure}  
In the graph of $\theta$,  there is local maximum at $\alpha=0.06144$ as shown in the Figure \ref{Fig 6.1}.\\
(ii) In Figure \ref{Fig 6.2} we fix $a=1$ and sketch the surface $(\alpha, b, \theta)$.\\
\begin{figure}[h]
   \begin{center}
             \includegraphics[width=0.7\textwidth]{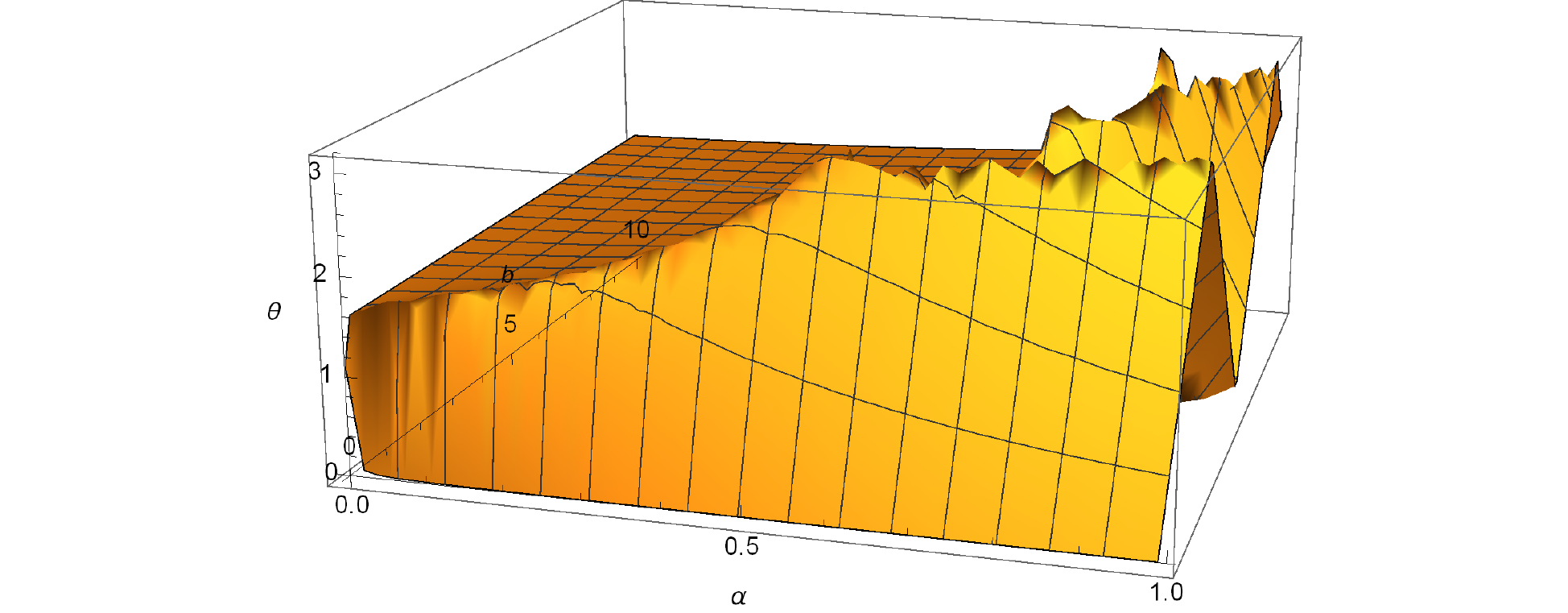}
            \caption{surface $(\alpha, b, \theta)$ for $0<\alpha<1$.}         
   \label{Fig 6.2}
   \end{center}  
          \end{figure}
It is observed that angle of rotation $\theta$ is having maxima at some values of parameters $b$ and $\alpha$.\\
In the Figure \ref{Fig 6.3}, we sketch a parametric curve of maximum ($\theta$) for different values of $\alpha$ and $b$.\\
\begin{figure}[h]
   \begin{center}
             \includegraphics[width=0.5\textwidth]{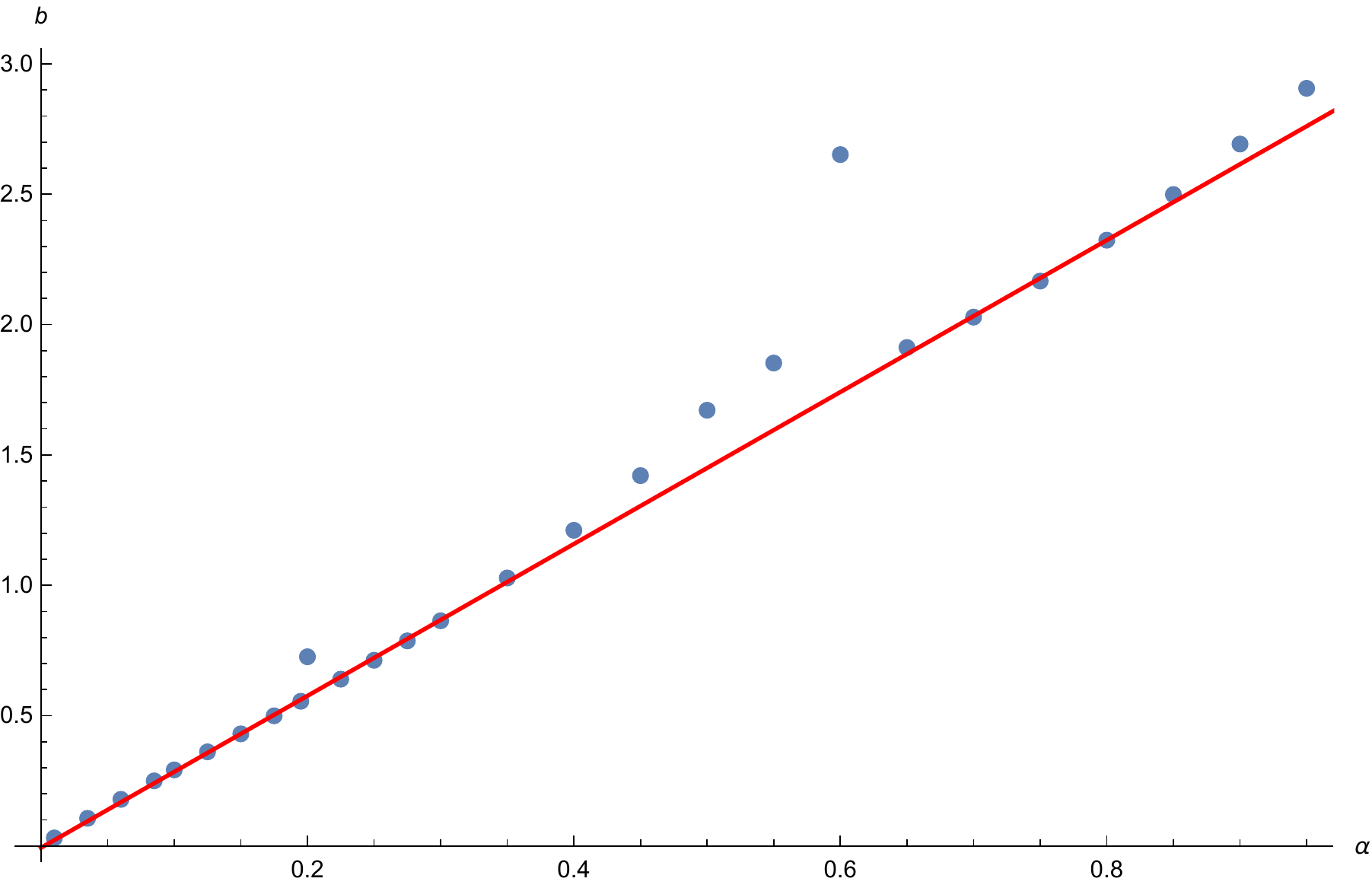}
            \caption{$max(\theta)$ shown by doted curve and it is approximated by a straight line (\ref{3.33}) (red color).}         
   \label{Fig 6.3}
   \end{center}  
          \end{figure}
It can be checked that most of the points of maximum ($\theta$) lie on a straight line \begin{equation}b=-0.0066+2.9128\alpha. \label{3.33}\end{equation}

\section{Conclusion}
The systems of fractional differential equations are not the dynamical systems in a classical sense. The solution $\phi_t(X_0)$ of fractional order initial value problem ${}_0^C\mathrm{D}_t^\alpha X=f(X)$, $X(0)=X_0$ does not satisfy the property $\phi_t\circ\phi_s=\phi_{t+s}$ of flow of classical differential equation. However, the two trajectories $\phi_t(X_0)$ and $\phi_{t+s}(X_0)$ are closely related if we take $f(X)=AX$, a linear function.
\par In this article, we have shown that the new trajectory is a linear transformation of original one. Further, we provided analysis of such trajectories with the help of Frenet appartus.
\section*{Acknowledgment}
S. Bhalekar acknowledges  the Science and Engineering Research Board (SERB), New Delhi, India for the Research Grant (Ref. MTR/2017/000068) under Mathematical Research Impact Centric Support (MATRICS) Scheme. M. Patil acknowledges Department of Science and Technology (DST), New Delhi, India for INSPIRE Fellowship (Code-IF170439).


\begin{thebibliography}{99}
\bibitem{Mainardi} F. Mainardi, Fractional Calculus and Waves in Linear Viscoelasticity: An Introduction to Mathematical Models, World Scientific, Singapore, (2010).    
 
\bibitem{Kulish} V. V. Kulish and Jos´e L. Lage Application of Fractional Calculus to Fluid Mechanics, J. Fluids Eng., {\bf124}(3), 803 (2002).  

\bibitem{Fellah} Z. E. A. Fellah, C.Depollier, Application of fractional calculus to the sound
   waves propagation in rigid porous materials: Validation via ultrasonic
   measurement, Acta Acustica, {\bf88}, 34--39 (2002).  
  
 \bibitem{Matusu} R. Matu\^s\.u, Application of fractional order calculus to control theory, International journal of mathematical models and methods in applied sciences, {\bf5}(7), 1162--1169 ( 2011).

\bibitem{El-Saka} H. A. A. El-Saka, S. Lee,  B. Jang, Dynamic analysis of fractional-order predator?prey biological economic system with Holling type II functional response, Nonlinear Dynamics, 1--10 (2019).
    
\bibitem{Yuan} L. Yuan, S. Zheng, Z. Alam, Dynamics analysis and cryptographic application of fractional logistic map, Nonlinear Dynamics, 1--22 (2019).

\bibitem{Goulart} A. G. O. Goulart, M. J. Lazo, J. M. S. Suarez, D. M. Moreira, Fractional derivative models for atmospheric dispersion of pollutants, Physica A: Statistical Mechanics and its Applications, {\bf477}, 9--19 (2017). 

   \bibitem{Sebaa} N. Sebaa, Z. E. A. Fellah, W. Lauriks, C. Depollier, Application of fractional calculus to ultrasonic wave propagation in human cancellous bone,
   Signal Processing archive, {\bf86}(10), 2668--2677  (2006).        

   \bibitem{Magin} R. L. Magin, Fractional Calculus in Bioengineering, Begell House, Redding, 269--355 (2006). 
 
\bibitem{Delbosco} D. Delbosco, L. Rodino, Existence and uniqueness for a nonlinear fractional differential equation, Journal of Mathematical Analysis and Applications, {\bf 204}(2), 609--625 (1996). 

  \bibitem{Diethelm} K. Diethelm, The Analysis of Fractional Differential Equations: An Application-Oriented Exposition Using Differential Operators of Caputo Type, Springer Science \& Business Media, New York, (2010).     

 \bibitem{Gejji1} V. Daftardar-Gejji, H. Jafari, Analysis of a system of nonautonomous fractional differential equations involving Caputo derivatives, Journal of Mathematical Analysis and Applications, {\bf328}(2), 1026--1033 ( 2007).   

\bibitem{Wei} Z. Wei, Q. Li, J. Che, Initial value problems for fractional differential equations involving Riemann-Liouville sequential fractional derivative, Journal of Mathematical Analysis and Applications, {\bf367}(1), 260--272 (2010).
  
 \bibitem{Matignon1} D. Matignon, ``Stability results for fractional differential equations with applications to control processing," Computational engineering in Systems and Application multiconference, IMACS, lille, france, {\bf 2}, 963--968 (1996). 
             
   \bibitem{Matignon2} D. Matignon, ``Stability properties for generalized fractional differential systems, ESAIM proceedings," {\bf 5}, 145--158 (1998).              
 
 \bibitem{Moze} M. Moze, J. Sabatier, A. Oustaloup, LMI tools for stability analysis of fractional systems, In ASME 2005 International Design Engineering Technical Conferences and Computers and Information in Engineering Conference American Society of Mechanical Engineers, 1611--1619 ( 2005).
               
    \bibitem{Deng1} W. Deng, C. Li, J. L\"{u}, ``Stability analysis of linear fractional differential system with multiple time delays," Nonlinear Dynamics, {\bf 48}, 409--416 (2007).
 
   \bibitem{Deng2} W. Deng, Smoothness and stability of the solutions for nonlinear fractional differential equations, Nonlinear Analysis: Theory, Methods \& Applications, {\bf72}(3-4), 1768-1777 (2010).
 
 \bibitem{Qian}  D. Qian, C. Li, R. P. Agarwal,  P. J. Wong,  Stability analysis of fractional differential system with Riemann-Liouville derivative, Mathematical and Computer Modelling, {\bf 52}(5-6), 862--874 (2010).
 
 \bibitem{Agarwal} R. Agarwal, D. O'Regan, S. Hristova, Stability of Caputo fractional differential equations by Lyapunov functions, Applications of Mathematics, {\bf 60}(6), 653--676 (2015).
  
 \bibitem{Zhang} S. Zhang, The existence of a positive solution for a nonlinear fractional differential equation, Journal of Mathematical Analysis and Applications, {\bf252}(2), 804--812 (2000).
 
\bibitem{Gejji2} V. Daftardar-Gejji,  Positive solutions of a system of non-autonomous fractional differential equations, Journal of Mathematical Analysis and Applications, {\bf 302}(1), 56--64 (2005).
 
\bibitem{Gejji3} A. Babakhani, V. Daftardar-Gejji, Existence of positive solutions of nonlinear fractional differential equations, Journal of Mathematical Analysis and Applications, {\bf 278}(2), 434--442 (2003).
 
 \bibitem{Bai} Z. Bai,  On positive solutions of a nonlocal fractional boundary value problem, Nonlinear Analysis: Theory, Methods \& Applications, {\bf72}(2), 916--924 (2010).
    
\bibitem{Goodrich} C. S. Goodrich, Existence of a positive solution to systems of differential equations of fractional order, Computers \& Mathematics with Applications, {\bf 62}(3), 1251--1268 ( 2011).

\bibitem{Baleanu} D. Baleanu, H. Mohammadi, S. Rezapour, Positive solutions of an initial value problem for nonlinear fractional differential equations, In Abstract and Applied Analysis Hindawi, {\bf 2012}, (2012).

\bibitem{Cong1} N. D. Cong, T. S. Doan, S. Siegmund, H. T. Tuan,  On stable manifolds for planar fractional differential equations. Applied mathematics and Computation, {\bf 226}, 157--168 (2014).
 
\bibitem{Deshpande} A. Deshpande,   V. Daftardar-Gejji, Local stable manifold theorem for fractional systems, Nonlinear Dynamics, {\bf83}(4), 2435--2452 (2016).
 
 \bibitem{Cong2} N. D. Cong, T. S. Doan, S.  Siegmund,  H. T. Tuan, On stable manifolds for fractional differential equations in high-dimensional spaces, Nonlinear Dynamics, {\bf 86}(3), 1885--1894 (2016).  
        
 \bibitem{S. Bhalekar} S. Bhalekar, M.  Patil,  Singular points in the solution trajectories of fractional order dynamical systems, Chaos: An Interdisciplinary Journal of Nonlinear Science, {\bf28}(11), 113123 (2018).  
                    
 \bibitem{Podlubny} I. Podlubny, Fractional Differential Equations, Academic Press, New York, (1999).     

\bibitem{O'Neill} B. O'Neill, Elementary Differential Geometry, Academic Press, New York, (1966).

          \bibitem{Luchko} Y. Luchko, R. Gorenflo, An operational method for solving fractional differential equations with the Caputo derivatives, Acta Math. Vietnam., {\bf24} 207--233 (1999).
   
      \end{thebibliography}
\end{document}